\documentclass[11pt]{amsart}
\usepackage{mathabx}
\usepackage{graphicx}
\usepackage{subfig}
\usepackage{tabularx}
\usepackage{wasysym,mathtools}

\usepackage{mathrsfs}     
\usepackage{helvet}         
\usepackage{courier}        
\usepackage{type1cm}      
\usepackage{color}
\usepackage{url}

\usepackage{geometry,calc,color}
\usepackage{amsmath,amssymb}

\usepackage{enumerate}
\usepackage{arydshln}
\usepackage{siunitx}
\usepackage{booktabs}
\usepackage{multirow}
\usepackage{tikz}
\usepackage{pgfplots}
\pgfplotsset{compat=1.9}

\usepackage{textcomp}
\usepackage[absolute,overlay]{textpos}
\newcommand{\Omh}{\Omega_h}

\newcommand{\Th}{\mathcal{T}_h}

\newcommand{\EK}{\mathcal{E}^K}

\newcommand{\Pinabla}{\Pi^\nabla_K}

\newcommand{\Svem}{S_a^K}

\newcommand{\Vh}{V_h}

\DeclareMathOperator{\sinc}{sinc}

\newcommand{\Bh}{\mathcal{B}_{h,\gamma}}

\newcommand{\Ah}{\mathcal{A}_{h,\gamma}}
\newcommand{\Eb}{{\mathcal{E}^\partial }}
\newcommand{\Pnphi}{\Pi^\nabla\! (\phi)}
\newcommand{\Pnpsi}{\Pi^\nabla\!(\psi)}
\newcommand{\dn}{\partial_\nu}
\newcommand{\Oh}{\Omega_h}
\newcommand{\Pn}{\Pi^\nabla\!}

\newcommand{\Hoh}{\mathcal{H}(\Oh)}

\newtheorem{theorem}{Theorem}[section]
\newtheorem{corollary}[theorem]{Corollary}
\newtheorem{lemma}[theorem]{Lemma}

\theoremstyle{definition}

\newtheorem{assumption}[theorem]{Assumption}

\theoremstyle{remark}
\newtheorem{remark}[theorem]{Remark}

\numberwithin{equation}{section}

\newcommand{\roundPrecision}{2}

\definecolor{darkpastelgreen}{rgb}{0.01, 0.75, 0.24}
\definecolor{cadmiumgreen}{rgb}{0.0, 0.42, 0.24}
\title{High order VEM on curved domains.}

\author[S. Bertoluzza]{Silvia Bertoluzza}
\address{IMATI ``E. Magenes'', CNR, Pavia (Italy)}%
\email{silvia.bertoluzza@imati.cnr.it}%

\author[M. Pennacchio]{Micol Pennacchio}
\address{IMATI ``E. Magenes'', CNR, Pavia (Italy)}%
\email{micol.pennacchio@imati.cnr.it}%

\author[D. Prada]{Daniele Prada}
\address{IMATI ``E. Magenes'', CNR, Pavia (Italy)}%
\email{daniele.prada@imati.cnr.it}%

\thanks{{This paper has been realized in the framework of ERC Project CHANGE, which has received funding from the European Research Council (ERC) under the European Union’s Horizon 2020 research and innovation programme (grant agreement No 694515), and it was partially supported by INdAM - GNCS}}%

\begin{document}

\begin{textblock*}{15cm}(2.5cm,25.8cm) 
	\noindent
	\scriptsize{This is the postprint version of the paper published in {\it Rend. Lincei Mat. Appl.} Vol. 30, No. 3, pp. 391–412 (2019),  doi: {10.4171/RLM/853}. The original publication is available at https://www.ems-ph.org.
	}
\end{textblock*}

\begin{abstract}

We deal with the virtual element method (VEM) for solving the Poisson equation on a domain $\Omega$ with curved boundary. Given a polygonal approximation $\Omega_h$ of the domain $\Omega$, the standard order $m$ VEM~\cite{hitchVEM}, for $m$ increasing, leads to a suboptimal convergence rate. We adapt the approach of \cite{BDT} to VEM and we prove that an optimal convergence rate can be achieved by using a suitable correction depending on high order normal derivatives of the discrete solution at the boundary edges of $\Omega_h$, which, to retain computability, is evaluated after applying the projector $\Pi^\nabla$ onto the space of polynomials.   
Numerical experiments confirm the theory.
\end{abstract}

\maketitle


\section{Introduction}

The virtual element method (VEM) is a PDE discretization framework 
designed to easily handle meshes consisting of very general  polygonal or polyhedral elements \cite{basicVEM}. 
The method can be considered as  
a generalization of the Finite Element Method (FEM) to polytopal tessellations, in that it looks for the solution in a conforming discretization space with a Galerkin approach. By  giving up conformity in the discretization of the bilinear form corresponding to the differential operator, the method manages to avoid the explicit construction of the basis functions (whence the name {\em virtual}). Everything is computed directly in terms of the degrees of freedom by resorting to suitable ``computable'' (in terms of the degrees of freedom) elementwise projectors onto the space of polynomials (see \cite{hitchVEM}), ultimately allowing to define a discrete bilinear form satisfying polynomial exactness and stability properties which allow to prove  optimal error estimate for discretization of (arbitrary) order $m$.
Different model problems 
have already been tackled by using VEM (\cite{beirao_parab,Antonietti_VEM_Stokes,Antonietti_VEM_Cahn,beirao_stokes,beirao_Navier_Stokes,perugia_Helmholtz,beirao_elastic,beirao_linear_elasticity,VEM_3D_elasticity,VEM_mixed,VEM_discrete_fracture,VEM_Laplace_Beltrami}), and, while most of the literature deals with the $h$ version of the method, the $p$ and $hp$ versions were also discussed and analyzed (\cite{antonietti_p_VEM,beirao_hp,beirao_hp_exponential}).

In this paper we consider the problem of extending the method to problems in domains with smooth curved boundaries.  As it happens in the Finite Element case,  the approximation of the curved domain by straight facets introduces an error that, for higher order methods, can dominate the analysis. 
Different approaches for the accurate treatment of curved domains in the finite element framework can be found in literature, see e.g. \cite{FEMsurvey_curved}. Among the different possible approaches to such a problem, following the guidelines of \cite{BDT}, we choose here to approximate the curved domain $\Omega$ with a polygonal domain $\Oh$, while compensating for the discrepancy in the geometry by suitably modifying the bilinear form. Complying with the VEM philosophy, the modified bilinear form will retain the property of being computable in terms of the degrees of freedom. 
Of course, other approaches are possible. In \cite{VEM_curved_beirao} the authors propose a direct definition of a modified virtual element space that accommodates curved elements, whose boundary matches exactly the boundary of $\Omega$. While loosing exact reproduction of polynomials, the resulting method retains optimality and, contrary to the one we propose here, it is immediately well suited to deal with curved interior interfaces.  On the other hand, our approach has the advantage of only requiring, for the implementation, minor modifications with respect to the polygonal case. 
To the best of our knowledge, other numerical methods that can handle curved polytopal meshes are only \cite{Mimetic_curved} and \cite{HDG_curved}.

The paper focuses on a simple elliptic model problem in 2D and it is organized as follows. As the projection method of \cite{BDT} combines Nitsche's technique for imposing non homogeneous boundary conditions \cite{Nitsche} with the 
improved accuracy polygonal domain approximation of \cite{Thomee}, we start, in Section \ref{sec:notation}, to adapt Nitsche's method to the Virtual Element framework and we provide a theoretical analysis of the resulting discretization, proving stability and an error estimate. In Section \ref{sec:curvo} we introduce and analyze the discretization for the problem on curved domains, proving also in this case stability and optimal error estimate. Finally, in Section \ref{sec:expes}, we test the method on several test cases, with methods of different order.  Throughout the paper, we will use the notation $A \lesssim B$ (resp. $A \gtrsim B$) to signify that the quantity $A$ is bounded from above (resp. from below) by a constant $C$ times the quantity $B$, with $C$ independent of the mesh size parameter $h$, the diameter $h_K$ and the specific shape of the polygon $K$, but possibly depending on the polynomial order $m$ of the method, and on the shape regularity constant $\alpha_0$ and $\alpha_1$ appearing in Assumption \ref{shapereg}.


\newcommand{\PolyDom}{{\Omega}}

\section{The Nitsche's method in the Virtual Element Context}\label{sec:notation} Before considering the problem of solving a PDE on a domain with a curved boundary, let us discuss how Nitsche's method for imposing non  homogeneous boundary condition \cite{Nitsche} can be applied in the context of the virtual element method.
Throughout this section let then $\PolyDom$ denote a bounded polygonal domain. To fix the ideas, we consider the following simple model problem: 
\begin{equation}\label{prob_mod_PolyDom}
-\Delta u = f, \text{ in }\PolyDom, \qquad u = g, \text{ on }\partial \PolyDom,
\end{equation}
with $f \in L^2(\Omega)$ and $g \in H^{1/2}(\Omega)$. Assume that we are given a family of quasi uniform tessellations $\Th$ of $\PolyDom$ into  polygonal elements $K$ of diameter $h_K \simeq h$. 
We make the
following standard regularity assumptions on the polygons of the tessellation:
\begin{assumption}\label{shapereg}
	There exists constants $\alpha_0,\alpha_1 > 0$ such that:
	\begin{enumerate}[(i)]
		\item  each element $K \in \Th$ is star-shaped with respect to a ball of radius $\geq \alpha_0 h_K$;
		\item  for each element $K$  in $\Th$ the distance between any two vertices of $K$ is $\geq \alpha_1 h_K$.
			\end{enumerate}
\end{assumption}

Under Assumptions \ref{shapereg}, several bounds hold uniformly in $h_K$ \cite{Lipnikov}. In particular, in the following we will make use of an inverse inequality on the space $\mathbb{P}_m$ of  polynomials of order less than or equal to $m$: for all $p \in \mathbb{P}_m$ and for all $j,k$ with $0 \leq j \leq k$ it holds that
\begin{equation}\label{inversebase}
\| p \|_{k,K} \lesssim h_K^{j-k} \| p \|_{j,K}.
\end{equation}
Moreover we will make use of the following trace inequality: for all $\phi \in H^1(K)$ we have 
\begin{equation}\label{trace1}
\| \phi \|_{0,\partial K} \lesssim h_K^{-1/2} \| \phi \|_{0,K} + h_K^{1/2} | \phi |_{1,K}.
\end{equation}

\

We will consider the standard order $m$ Virtual Element discretization space (\cite{basicVEM}), whose definition we briefly recall.
For each  polygon $K \in \Th$ we let the space $\mathbb{B}_m(\partial K)$ be defined as
\begin{gather*}
\mathbb{B}_m(\partial K) = \{ v \in C^0(\partial K): v|_{e} \in \mathbb{P}_m\ \forall e \in \EK \},
\end{gather*}
where 
$\EK$ denotes 
the set of edges of the polygon $K$.   We introduce the local VE space as: 
\[
V^{K,m} =   \{ v \in H^1(K):\ v|_{\partial K} \in \mathbb{B}_m(\partial K),\ \Delta v \in \mathbb{P}_{m-2} \}
\]
(with $\mathbb{P}_{-1}=\{0\}$). The global discrete VE space $V_h$ is then defined as
\begin{align}\label{VEMm}
V_h &= \{ v \in H^1(\PolyDom): v|_{K} \in V^{K,m}\ \forall K \in \Th \} =\\
& = \{ v \in H^1(\PolyDom):\forall K \in \Th\  v|_{\partial K} \in \mathbb{B}_m(\partial K),\ \Delta v|_K \in  \mathbb{P}_{m-2} \}. \nonumber 
\end{align}
A function in $V_h$ is uniquely determined by the following degrees of freedom 
\begin{itemize}
	\item its  values at the vertices of the tessellation;
	\item (only for $m \geq 2$) for each edge $e$, its values at the $m-1$ internal points of the $m+1$-points Gauss-Lobatto quadrature rule on $e$;
	\item (only for $m \geq 2$)  for each element $K$, its moments in $K$ up to order $m-2$. 
\end{itemize} 
For any given function $w \in H^2(\Omega)$ we can then define the unique function $w_I \in V_h$ such that: a) the values of $w$ and $w_I$ at the vertices of the tessellation coincide;
b) for each edge $e$, the values of $w$ and $w_I$  at the $m-1$ internal points of the $m+1$-points Gauss-Lobatto quadrature rule on $e$ coincide;
c) for each element $K$, the moments up to order $m-2$ of $w$ and $w_I$  in $K$ coincide. 
The function $w_I$ satisfies the following local approximation bound \cite{basicVEM}: if $w \in H^s(K)$, with $2 \leq s \leq m+1$, then
\begin{equation}\label{VEMapproxI}
\| w - w_I \|_{0,K} + h_K | w - w_I |_{1,K} \lesssim h_K^s | w |_{s,K}.
\end{equation}

\
\newcommand{\pwPoly}{\mathbb{P}^*_m}

Let  now
\[a(\phi,\psi)= \int_\PolyDom \nabla \phi \cdot \nabla \psi, \qquad a^K(\phi,\psi) = \int_K\nabla \phi \cdot \nabla \psi,\]
 and let $\Pinabla : H^1(K) \rightarrow \mathbb{P}_{m}(K)$ denote the projection defined by
\[
a^K(\Pinabla \phi, p) = a^K(\phi,p), \quad \forall p \in \mathbb{P}_m, \qquad \int_K \Pinabla \phi = \int_K \phi.
\]
We recall that for $w_h \in V^{K,m}$, $\Pinabla w_h$ can be computed directly from the values of the degrees of freedom (\cite{hitchVEM}) without the need of explicitly constructing it (which would imply somehow solving a partial differential equation), by taking advantage of the identity 
\[
\int_{K}\nabla w_h \cdot \nabla p = - \int_{K} w_h \Delta p + \int_{\partial K} w_h \frac{\partial p}{\partial \nu_K},
\]
that allows to express the term of the left hand side in terms of the interior moments of $w_h$ (for $p \in \mathbb{P}_m$, $\Delta p$ is a polynomial of degree less than or equal to $m-2$) and of an integral on the boundary (where $w_h$ is a known piecewise polynomial).
  Letting $\pwPoly$ denote the space of discontinuous piecewise polynomials of order less than or equal to $m$
\[
\pwPoly = \{\phi \in L^2(\Omega): \phi|_K \in\mathbb{P}_m\ \forall K \in \Th\},
\]
 we let
 $\Pi^\nabla : H^1(\Omega) \to \pwPoly$ be defined by assembling, element by element, the $\Pinabla$'s:
 \[
 \Pi^\nabla \phi|_K = \Pinabla (\phi|_K), \ \forall K \in \Th.
 \]

The discretization of \eqref{prob_mod_PolyDom} by the Nitsche's method would consist in looking for $u_h \in V_h$ such that for all $v_h \in V_h$ one has
\begin{align*}
a(u_h,v_h) - \int_{\partial\PolyDom} \partial_\nu u_h v_h - \int_{\partial\PolyDom} u_h\partial_\nu v_h + \gamma h^{-1} \int_{\partial\PolyDom} u_h v_h =\\
\int_{\PolyDom}f v_h  - \int_{\partial\PolyDom} g \partial_\nu v_h  + \gamma h^{-1} \int_{\partial\PolyDom} g v_h,
\end{align*}
where $\partial_\nu$ stands for ${\partial}/{\partial \nu}$, $\nu$ denoting the outer normal to $\Omega$. As typical for the Virtual Element method, both at the right hand side and at the left hand side of such an equation we find terms which are not ``computable", that is that can not be computed exactly with only the knowledge of the value of the degrees of freedom of $u_h$ and $v_h$. Besides the bilinear form $a$, which can be treated by the standard approach,  this is the case for  all the terms involving $\partial_\nu u_h$ and $\partial_\nu v_h$. 
As usually done in VEM, 
the bilinear form $a$
  is then replaced  with 
\[
a_h(u_h,v_h) = \sum_K a_h^K(u_h,v_h),
\]
where 
\[
a^{K}_h(\phi,\psi) = a^K(\Pinabla  \phi,\Pinabla  \psi) + \Svem(\phi - \Pinabla  \phi, \psi - \Pinabla  \psi).
\]
We recall that different choices are possible for the bilinear form $\Svem$ (see \cite{beirao_stab}),
the essential requirement being that it satisfies
$$  a^K(\phi,\phi) \lesssim  \Svem(\phi,\phi) \lesssim a^K(\phi,\phi),\quad \forall \phi \in V^{K,m}\ \text{ with } \Pinabla  \phi=0,
$$
 so that 
the local discrete bilinear forms satisfy the
following two properties:
\begin{itemize}
\item {\em Stability}: 
$$
   a^K(\phi,\phi) \lesssim   a^K_h(\phi,\phi) \lesssim a^K(\phi,\phi) ,\quad \forall \phi \in V^{K,m}
$$
\item {\em m-consistency}: for any $\phi\in \Vh$ and $p\in \mathbb{P}_{m}(K)$ 
\begin{equation}\label{mconsistency}
a_h^K(\phi,p) = a^K(\phi,p). 
\end{equation}
\end{itemize}
In the numerical tests performed in Section \ref{sec:expes} we made the standard choice of defining $\Svem$ in terms of the vectors of local degrees of freedom as the properly scaled euclidean scalar product. 

\


%
\renewcommand{\epsilon}{\varepsilon}

As far as the terms involving the normal derivative are concerned, we treat them by replacing $\partial_\nu u_h$ and $\partial_\nu v_h$,  boundary edge by boundary edge,  respectively with 
 $\partial_\nu \Pi^\nabla(u_h)$ and $\partial_\nu \Pi^\nabla(v_h)$. 
 Then we can write the Nitsche's method for the VEM discretization of \ref{prob_mod}  as:{ find $u_h \in V_h$ such that for all $v_h \in V_h$ it holds that}
 \begin{multline}\label{problem_Nitsche}
 a_h(u_h,v_h) - \sum_{e\in \Eb} \int_e \dn \Pn(u_h) v_h- \sum_{e\in \Eb} \int_e \dn\Pn(v_h) u_h + \gamma h^{-1} \int_{\partial \PolyDom} u_h v_h \\ = \int_{\PolyDom} f v_h - \sum_{e \in \Eb} \int_e  g \left( \dn \Pn (v_h) - \gamma h^{-1} v_h \right),
 \end{multline}
 where $\gamma$ is a positive constant and $\mathcal{E}^\partial$ denotes the set of edges of $\Th$ lying on $\partial\Omega$.
 
 \
 
We introduce the norm:
\begin{equation}\label{norm_en}
\vvvert \phi \vvvert ^2_\PolyDom = |\phi |^2_{1,\PolyDom} + h^{-1} \| \phi \|^2_{0,\partial\PolyDom}
\end{equation}
and the space $\mathcal{H}(\PolyDom)$ defined as the closure of $C^{\infty}(\PolyDom)$ with respect to the norm $\vvvert \cdot \vvvert_\PolyDom$. Setting
\begin{align}\label{bilinear_Nitsche}
\Bh (\phi,\psi) &= a_h(\phi,\psi) - \sum_{e\in \Eb} \int_e \dn \Pnphi\psi \\ 
&- \sum_{e\in \Eb} \int_e \dn\Pnpsi \phi + \gamma h^{-1} \int_{\partial \PolyDom} \phi \psi, \nonumber
\end{align}
we start by proving the following lemma:
\begin{lemma}\label{lem:coercivity_Nitsche} For all $\phi$, $\psi \in \mathcal{H}(\PolyDom)$ we have 
		\begin{equation}\label{continuity0}	| \Bh(\phi,\psi)| \lesssim \vvvert \phi \vvvert_{\PolyDom}\, \vvvert \psi \vvvert_{\PolyDom}.\end{equation}	
	Moreover, there exists $\gamma_0 >0$  
	 such that, for $\gamma > \gamma_0$,
	 the bilinear form $\Bh$ verifies for all $\phi \in V_h$
	\begin{equation}\label{coercivity0}
	\Bh(\phi,\phi) \gtrsim 
	\vvvert \phi \vvvert_\PolyDom ^2,
	\end{equation}
	(the implicit constant in the two inequalities depending on $\gamma$).
\end{lemma}

\begin{proof}
We observe that, since $\Pnphi$ is a polynomial, it is not difficult to verify that the following inverse bound holds:
\begin{equation}\label{inverse0}
\| \dn \Pnphi \|_{0,e} \lesssim h^{-1/2} | \Pnphi |_{1,K_e} \leq h^{-1/2} | \phi |_{1,K_e},
\end{equation}
where, for $e \in \Eb$,  $K_e$ is the unique polygon of the tessellation having $e$ as an edge. Then	we have 
		\[
		\sum_{e\in \Eb} \int_e \dn \Pnphi\psi  \lesssim \sum_{e\in \Eb} \|  \dn \Pnphi \|_{0,e} \| \psi \|_{0,e} \lesssim  \sum_{e\in \Eb}  |  \phi |_{1,K_e} h^{-1/2} \| \psi \|_{0,e}.
		\]
		Obtaining \eqref{continuity0} is then not difficult.	
As far as \eqref{coercivity0} is concerned, we have
	\begin{equation}\label{firstbound}
	\Bh(\phi,\phi) = a_h(\phi,\phi) +  \gamma h^{-1} \| \phi \|_{0,\partial \PolyDom}^2 - 2 \langle \dn \Pnphi,\phi \rangle
	\end{equation}
	where we use the notation
	\[
	\langle \phi,\psi \rangle = \sum_{e \in \Eb} \int_e \phi \psi .
	\] 
	
	We now have the following bounds
\[
	\langle \dn \Pnphi, \phi \rangle \lesssim \sum_{e \in \Eb} \|  \dn \Pnphi \|_{0,e} \| \phi \|_{0,e}
	\]
Thanks to the inverse inequality (\ref{inverse0}), we can write, for $\varepsilon > 0$ arbitrary,
	\[
	\langle \dn \Pnphi, \phi \rangle \lesssim \sum_{e \in \Eb} h^{-1/2} | \phi |_{1,K_e} \| \phi \|_{0,e}
	\lesssim \frac \varepsilon 2 | \phi |^2_{1,\PolyDom} + \frac 1 {2\varepsilon} h^{-1} \| \phi \|^2_{0,\partial \PolyDom}. 	\]

	Substituting into \eqref{firstbound} we obtain, for a fixed positive constant $c_1$ independent of $h$
	\[
	\Bh(\phi,\phi) \gtrsim (1 - c_1  \varepsilon) | \phi |_{1,\PolyDom}^2 +( \gamma - \frac {c_1} \epsilon )  h^{-1} \| \phi \|^2_{0,\partial \PolyDom}.
	\]
	We now choose $\varepsilon = 1/(2c_1)$ and if $\gamma > \gamma_0$ with $\gamma_0$ chosen  in such a way that $\gamma_0 - c_1/\varepsilon > 0$,
	the thesis easily follows.
\end{proof}

Existence and uniqueness of the solution of (\ref{problem_Nitsche}) easily follow.

\

We are then able to prove the following result:
\smallskip

\begin{theorem} \label{theo:error_Nitsche}
	If  $u\in H^s(\Omega)$, with $2 \leq s \leq m+1$, and if we  chose $\gamma >\gamma_0$, with $\gamma_0$  given by Lemma \ref{lem:coercivity_Nitsche}, then the following error estimate holds
	\[
	\vvvert u  - u_h \vvvert_{\PolyDom} \lesssim h^{s-1} | u |_{s,\PolyDom} .
	\]
\end{theorem}

\smallskip

\begin{proof} Let $u_I$ denote the  VEM interpolant and $u_\pi \in \pwPoly$ the $L^2(\Omega)$ projection of $u$ onto the space of discontinuous piecewise polynomials. 
For any $j$, $k$ with $0\leq j \leq k \leq m+1$ we have, for $u \in L^2(\Omega)$ with $u|_K \in H^k(K)$
\begin{equation}\label{VEMapproxpi}
\| w - w_\pi \|_{j,K} \lesssim h_K^{k-j} | w |_{k,K}.
\end{equation}
For $j \geq 1$ the bound \eqref{VEMapproxpi} can be proven by a standard argument combining the bound for $j = 0$ with an inverse inequality. Moreover we have, for $e \in \mathcal{E}^K \cap \Eb$,
\begin{equation}\label{localPoly}
\| \dn(w - w_\pi) \|_{0,e} \lesssim h^{-1/2} | w - w_\pi |_{1,K} + h^{1/2} | w - w_\pi |_{2,K} \lesssim h^{s - 3/2} | w |_{s,K}.
\end{equation}
We set $d_h = u_I-u_h$. Using the definition (\ref{bilinear_Nitsche}), summing and subtracting $u_\pi$, using  the $m$-consistency (\ref{mconsistency}) to replace $a_h(u_\pi,d_h)$ with $a(u_\pi,d_h)$ and then, summing and subtracting $u$,  we can write
	\begin{multline}
	\vvvert u_I - u_h \vvvert_\PolyDom^2 \lesssim \Bh(u_I, d_h) - \Bh(u_h,d_h) =	\\	 = \sum_{K\in\Th} a_h^K(u_I-u_\pi,d_h) + \sum_{K\in \Th} a^K(u_\pi - u,d_h) +  a(u,d_h) 	- \langle \dn u, d_h \rangle + \\ 
 \langle \dn(u - \Pn (u_I)),d_h \rangle - \langle u , \dn \Pn (d_h)	\rangle + \langle u - u_I, \dn  \Pn (d_h)	 \rangle  + \gamma h^{-1} \langle u,d_h \rangle 
	\\+ \gamma h^{-1}  \langle u_I-u,d_h \rangle 
	 - \int_{\PolyDom} f d_h + \langle  g, \dn \Pn(d_h) - \gamma h^{-1} d_h \rangle \\= E1 + E2 + E3 + E4  + E5 
\end{multline}		
	with
	\begin{gather*}
	E1 = \sum_{K\in\Th} a_h^K(u_I-u_\pi,d_h), \quad E2 = \sum_{K\in\Th} a^K(u_\pi - u,d_h), \quad \\
	E3 =  \langle \dn(u - \Pn (u_I)),d_h \rangle \quad 
	E4 = \langle u - u_I, \dn  \Pn (d_h) \rangle,	\\
	  E5 =  \gamma h^{-1}\langle u_I-u,   d_h\rangle,
	\end{gather*}
	where we used that, as $u$ is the solution of \eqref{prob_mod_PolyDom}, 
	\begin{gather*}
	a(u,d_h) - \langle \dn u, d_h \rangle - \int_{\PolyDom} f d_h = 0, 
	\\ \langle u , \dn \Pn (d_h)	\rangle - \gamma h^{-1} \langle u,d_h \rangle -  \langle  g, \dn \Pn(d_h) - \gamma h^{-1} d_h \rangle  = 0.
	\end{gather*}
	Let us then bound the different components of the error.
	
	\
	
	Both $E1$ and $E2$ are standardly encountered in the analysis of the  VEM method, and a bound can be found in the literature (see e.g. \cite{basicVEM}), yielding
	\[
	E1 \lesssim | d_h |_{1,\PolyDom} h^{s-1} | u |_{s,\PolyDom}, \quad \text{ and }\quad E2 \lesssim | d_h |_{1,\PolyDom} h^{s-1} | u |_{s,\PolyDom}.
	\]
	On the other hand we have
	\[
	E3 \lesssim  \sum_{e \in \Eb} \| \dn(u - \Pn (u_I)) \|_{0,e} \| d_h  \|_{0,e} .
	\]
	Now, using \eqref{inverse0} and \eqref{localPoly}, we have
	\begin{align*}
	\| \dn(u - \Pn u_I) \|_{0,e} \leq \| \dn(u - u_\pi) \|_{0,e} + \| \dn \Pn( u_\pi - u_I) \|_{0,e} \lesssim \\
	h^{s-3/2} | u |_{s,K} + h^{-1/2} (| u_\pi - u_I |_{1,K} + |u_I - u |_{1,K}),
	\end{align*}
	yielding
	\[
	E3 \lesssim h^{s-1} h^{-1/2} | u |_{s,\PolyDom} \| d_h \|_{0,\partial \PolyDom}.
	\]
	As far as $E4$ is concerned we have
	\begin{align*}
	E4 \leq \sum_{e \in \Eb} \| u - u_I \|_{0,e} \|  \dn  \Pn (d_h)	 \|_{0,e} \lesssim  \sum_{e \in \Eb} h^{s-1/2} | u |_{s,K_e} h^{-1/2}\,
	 |d_h |_{1,K_e} \lesssim \\
	 h^{s-1} | u |_{s,\PolyDom} | d_h |_{1,\PolyDom}.
	\end{align*}
	A similar argument yields
	\[
	E5 \lesssim h^{s-1} | u |_{s,\PolyDom} h^{-1/2} \| d_h \|_{0,\partial\PolyDom},
	\]
finally giving
\[
\vvvert u_I - u_h \vvvert_{\Omega}^2 \lesssim h^{s-1} | u |_{s,\Omega} \vvvert u_I - u_h \vvvert_\Omega.
\]
	Dividing both sides by $\vvvert u_I - u_h \vvvert_\PolyDom$ and using a triangular inequality we get the thesis.
\end{proof}


\newcommand{\dnh}{\partial_{\nu_h}}
\newcommand{\tg}{g^\star}

\section{The Virtual Element Method on domains with curved boundary}\label{sec:curvo}

Let now consider the solution of the same model problem
\begin{equation}\label{prob_mod}
	-\Delta u = f, \text{ in }\Omega, \qquad u = g, \text{ on }\partial \Omega
\end{equation}
with, once again, $f \in L^2(\Omega)$, $g \in H^{1/2}(\partial \Omega)$, where now $\Omega \subseteq \mathbb{R}^2$ is a convex domain with curved boundary $\partial\Omega$ assumed, for the sake of convenience, to be of class $C^\infty$. In order to solve such a problem by the Virtual Element method, 
we assume that $\Omega$ is approximated by a family of polygonal  domains $\Omh$, $0 < h \leq 1$,  each endowed with a quasi uniform shape regular tessellation $\Th$  into 
polygons $K$ with diameter $h_K \simeq h$, see Fig. \ref{fig:domain}.  We assume that all the vertices of $\Th$ lying  on $\partial \Omh$ also lie on $\partial \Omega$. As $\Omega$ is convex this implies that $\Omh \subseteq \Omega$.

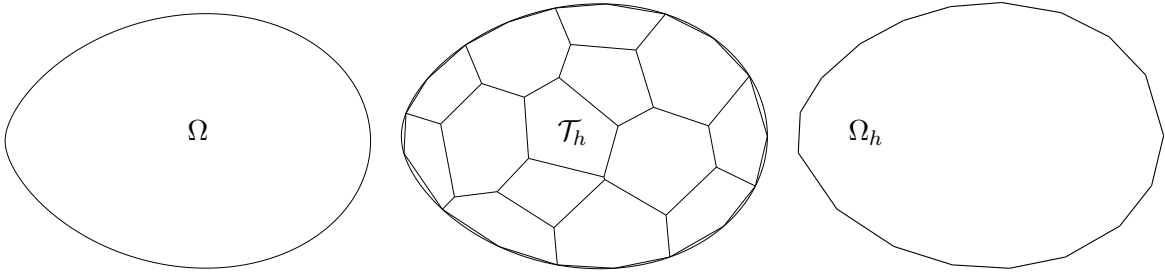
\begin{figure}[htb]
	\centering
	\begin{picture}(200,100)
	\put(-120,0){\includegraphics[width=0.3\textwidth]{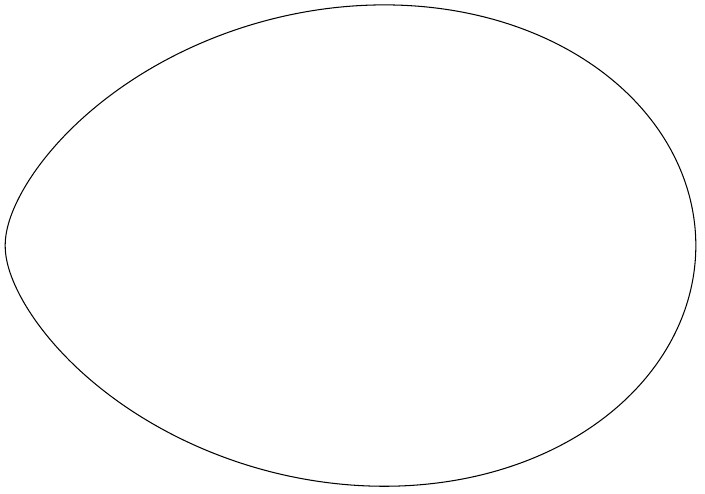}}
	\put(30,0){\includegraphics[width=0.3\textwidth]{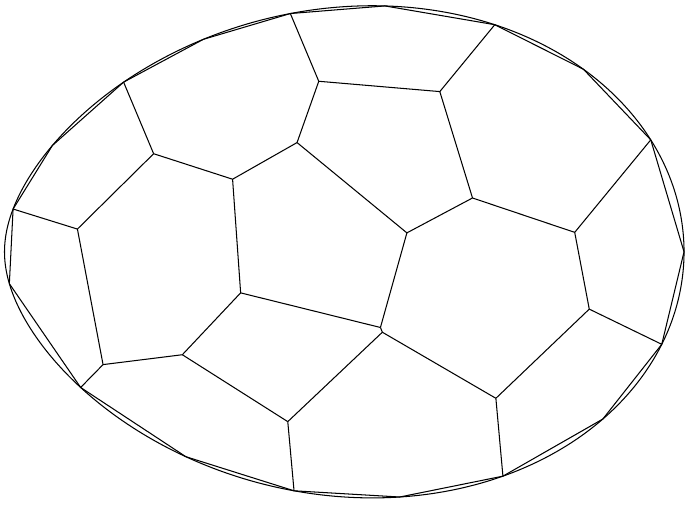}}
	\put(180,0){\includegraphics[width=0.3\textwidth]{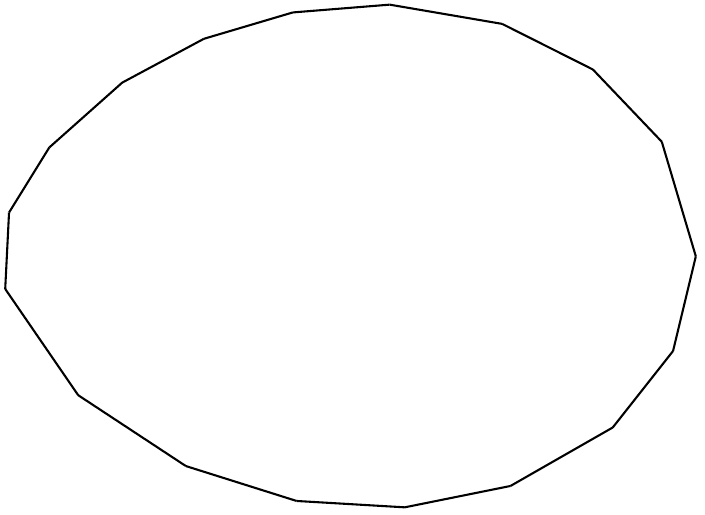}}
	\put(-50,50){$\Omega$}
	\put(90,50){$\Th$}
	\put(200,50){$\Omh$}
	\end{picture}
		\caption{Domain $\Omega$,  tessellation $\Th$ of  
		$\Omega $ into shape regular polygons and polygonal domain $\Omh \subset \Omega$ approximating $\Omega$. %
	}
	\label{fig:domain}
\end{figure}

\newcommand{\nuhx}{\nu_h}

\newcommand{\supd}{\delta_h}

\subsection{The Projection Method of Bramble, Dupont and Thom\'ee}\label{sec:BDT}
In order to deal with the curved boundary, following \cite{BDT} we apply  on the approximating domain $\Oh$ a modified version of Nitsche's method, which takes into account that the boundary data is given on $\partial \Omega$ rather than on $\partial \Oh$.

Letting $\nuhx$ be the outer normal to $\Omh$, for $x \in \partial\Oh$ we let $\delta(x) > 0$ denote the non negative scalar such that 
\[
x + \delta(x) \nuhx(x) \in \partial \Omega.
\]
It is known (\cite{BDT}) that, as $\Omega$ is smooth and convex, we have  that 
\begin{equation}
\supd = \sup_{x \in \partial\Omh} \delta(x)  = o(h^2).
\end{equation}\label{delta_h}

\

We let $V_h \subset H^1(\Oh)$  be the order $m$ VEM discretization space relative to the tessellation $\Th$ (defined by (\ref{VEMm})).  Setting $k =\lfloor m/2\rfloor$,  the projection method of \cite{BDT} reads as follows: 
{ find $u_h \in V_h$ such that for all $v_h \in V_h$ it holds that}
\begin{multline}\label{problem_BDT}
 \Bh(u_h,v_h)  -   \sum_{e\in \Eb }
 \int_e 
\left(\sum_{j=1}^k \frac{\delta^j}{j!} \dnh^j \Pn(u_h)\right)
 ( \dnh \Pn(v_h)  - \gamma h^{-1} v_h)  \\ = \int_{\Oh} f v_h -
\sum_{e \in \Eb} \int_e \tg  \left( \dnh \Pn (v_h) - \gamma h^{-1} v_h \right)
\end{multline}
where $\dnh^j u = ({\partial}_{\nu_h})^j u$ denotes the $j$-th partial derivative of $u$ in the $\nu_h$ direction and where, for $x \in \partial\Oh$, \[\tg (x) = g(x+\delta(x) \nuhx(x)).\]

\

We then introduce the bilinear form $\Ah$ defined by
\begin{align}
\Ah (\phi,\psi) = \Bh(\phi,\psi)  -  
\sum_{e\in \Eb }
\int_e 
\left(\sum_{j=1}^k \frac{\delta^j}{j!} \dnh^j \Pn(u_h)\right)
( \dnh \Pn(v_h)  - \gamma h^{-1} v_h),
\end{align}
and prove the following lemma:
\begin{lemma}\label{lem:coercivity}
	For all $\phi$, $\psi$ in $\Hoh$ it holds that
	\begin{equation}\label{continuity}	| \Ah(\phi,\psi) \lesssim \vvvert \phi \vvvert_{\Oh}\, \vvvert \psi \vvvert_{\Oh}.\end{equation}	
	Moreover, there exists $\gamma_0 >0$ such that for all $\gamma > 0$  the bilinear form $\Ah$ verifies for all $\phi \in V_h$
	\begin{equation}\label{coercivity}
	\Ah(\phi,\phi) \gtrsim %
	\vvvert \phi \vvvert^2_{\Oh}
	\end{equation}	 
	provided $h < h_0$ with $h_0 = h_0(\gamma) > 0$.
\end{lemma}

\newcommand{\Ck}{\mathcal{C}_k}

\begin{proof}
	Let $\Ck$ be defined as
	\[
	\Ck(\phi,\psi) = \sum_{e\in \Eb }\int_e \left(
	\sum_{j=1}^k \frac{\delta^j}{j!} \dnh^j \phi
	\right) \psi 
	\]
so that we can write
	\[
	\Ah(\phi,\phi) = 
	\Bh(\phi,\psi) 
	- \Ck(\Pnphi,\dnh \Pnpsi - \gamma h^{-1} \psi).
	\]
		As, by definition, $\Pinabla (\phi)$ is a piecewise polynomial, using an inverse inequality and the continuity of the operator $\Pinabla$ we get that for $e \in \mathcal{E}^\partial$ and $\phi \in H^1(\Omega)$ we have
			\begin{equation}\label{inverse}
			\| \dnh^j \Pinabla( \phi )\|_{0,e} \lesssim h^{1/2-j} | \Pinabla (\phi) |_{1,K_e} \lesssim  h^{1/2-j} |  \phi |_{1,K_e},
			\end{equation}
	where, once again, $K_e$ is the unique element of $\Th$ having $e$ as an edge.	Since $\supd/h \lesssim 1$, using \eqref{inverse} we have
		\begin{multline}\label{boundCk1}
		| \Ck(\Pnphi,\psi) | \lesssim  \sum_{e\in \Eb } 
		\sum_{j=1}^k \left(\frac \supd h \right)^j h^j \| \dnh^j \Pnphi \|_{0,e} \|  \psi \|_{0,e}   \lesssim \frac \supd h
				 | \phi |_{1,\Oh} h^{1/2}  \| \psi \|_{0,e} 
		\end{multline}
		as well as
	\begin{multline}\label{boundCk2}
		| \Ck(\Pnphi,\dnh \Pnpsi) | \lesssim  \sum_{e\in \Eb } 
		\sum_{j=1}^k \left(\frac  \supd h \right)^j h^j \| \dnh^j \Pnphi \|_{0,e} \|  \dnh \Pnpsi \|_{0,e}  
		 \lesssim \frac \supd h | \phi |_{1,\Oh} | \psi |_{1,\Oh}. 
	\end{multline}
		Combining with \eqref{continuity0} the bound \eqref{continuity} easily follows.
Let us now consider \eqref{coercivity}.	
Combining \eqref{coercivity0} with \eqref{boundCk1} and \eqref{boundCk2} we obtain, for $\varepsilon >0$ arbitrary and $c_1$, $c_2$ and $c_3$ fixed positive constants,
	\begin{multline*}
	\Ah(\phi,\phi) = 
	\Bh(\phi,\phi) 
	- \Ck(\Pnphi,\dnh \Pnphi) + \gamma h^{-1} \Ck(\Pnphi,\phi)\\ \gtrsim  (1 - c_1  \varepsilon- c_2 \frac  \supd h - c_3 \frac  \supd h \gamma) | \phi |_{1,\Oh}^2 +( \gamma - \frac {c_1} \epsilon - c_3 \gamma \frac \supd h )  h^{-1} \| \phi \|^2_{0,\partial \Oh} .
	\end{multline*}

	We now choose $\varepsilon = 1/(2c_1)$ and we fix $\gamma_0$ in such a way that $\gamma_0 - c_1/\varepsilon > 0$. For $\gamma > \gamma_0$, set now $\alpha = \gamma - c_1/\varepsilon > 0$. As $ \supd = o(h^2)$, we can choose $h_0$ in such a way that for all $h < h_0$,  $c_2 \supd/h + \gamma c_3 \supd/h < 1/2$ and $\gamma c_3  \supd/h < \alpha$. The thesis easily follows.
\end{proof}

Once again, existence and uniqueness of the solution of (\ref{problem_BDT}) easily follow. Moreover, the error estimate given by the following Theorem holds.

\begin{theorem}\label{main_theo} If $u\in H^{s}(\Omega)\cap W^{k+1,\infty} (\Omega)$, with $k+1 < s \leq m+1$, for $h< h_0$ and $\gamma > \gamma_0$ ($\gamma >0$ and $h_0 >0$ given by Lemma \ref{lem:coercivity}) the following error estimate holds
	\[
	\vvvert u  - u_h \vvvert_{\Oh} \lesssim h^{s-1} | u |_{s,\Omega} + h^{-1/2} \delta^{k+1} \| u \|_{k+1,\infty,\Omega}.
	\]
\end{theorem}

\begin{proof} Let $u_I$ denote the usual VEM interpolant. 
	Setting $d_h = u_I-u_h$ and proceeding as in the proof of Theorem \ref{theo:error_Nitsche}, we have
	\begin{multline}
	\vvvert u_I - u_h \vvvert^2_{\Oh} \lesssim \Ah(u_I, d_h) - \Ah(u_h,d_h) =	\\	  \sum_{K\in\Th} a_h^K(u_I-u_\pi,d_h) + \sum_K a^K(u_\pi - u,d_h) +  a(u,d_h) 	- \langle \dnh u, d_h \rangle +\\ 
 \langle \dnh(u - \Pn (u_I)),d_h \rangle - \langle u , \dnh \Pn (d_h)	\rangle + \langle u - u_I, \dnh  \Pn (d_h)	 \rangle  +  \gamma h^{-1}\langle u,d_h \rangle 
	\\+ \gamma h^{-1} \langle u_I-u,d_h \rangle 	- \Ck(u,\dnh \Pn(d_h) - \gamma h^{-1} d_h ) + \\
	+\Ck (u - \Pn(u_I), \dnh \Pn(d_h) - \gamma h^{-1} d_h)  - \int_{\Oh} f d_h + \langle \tg , \dnh \Pn(d_h) - \gamma h^{-1} d_h \rangle \\= E1 + E2 + E3 + E4  + E5 + E6 + E7 
\end{multline}			
	with \(E1\), \(E2\), \(E3\), \(E4\),  and \(E5 \) as in the proof of Theorem   \ref{theo:error_Nitsche}
	and with
	\begin{align*} E6 = \Ck (u - \Pn(u_I), \dnh \Pn(d_h) - \gamma h^{-1}d_h), 
	\\
	 E7 = \langle \tg  - \sum_{j=0}^{k} \frac {\delta^j}{j!} \dnh^j u, \dnh\Pn(d_h) - \gamma h^{-1} d_h \rangle.
	\end{align*}
The components $E1, E2, E3, E4, E5$ can be bounded exactly as in Theorem \ref{theo:error_Nitsche}. 
Let us then bound  the last two terms $E6$ and $E7$.
	Since $\delta = o(h^2) \lesssim h$, we have
	\[
	E6 \lesssim \sum_{e \in \Eb} \sum_{j=1}^k h^j \| \dnh^j (u - \Pn(u_I)) \|_{0,e} (\| \dnh \Pn(d_h)  \|_{0,e} + h^{-1} \| d_h \|_{0,e}).
	\]
	Now, using \eqref{trace1} as well as \eqref{inverse} we have
	\begin{multline*}
	\| \dnh^j (u - \Pn(u_I)) \|_{0,e}  \lesssim 
	\| \dnh^j (u - u_\pi) \|_{0,e} + \| \dnh^j (\Pn(u_\pi - u_I)) \|_{0,e} 
\\	\lesssim h^{-1/2} | u - u_\pi |_{j,K_e} + h^{1/2} | u - u_\pi |_{j+1,K_e} + h^{1/2-j} | u_\pi - u_I |_{1,K_e} \lesssim
h^{s-j-1/2} | u |_{s,K_e},
\end{multline*}
where $| u_\pi - u_I |_{1,K_e}$ is bound by adding and subtracting $u$ and using the VEM approximation bounds \eqref{VEMapproxI} and \eqref{VEMapproxpi}, yielding
\[
E6 \lesssim h^{s-1} | u |_{s,\Omega} \vvvert d_h \vvvert_{\Oh}.
\]

\

	Finally, $E7$ takes into account the approximation of the curved boundary by projection, and, following the paper by Thom\'ee, it can be bound as
	\begin{multline*}
	E7 \lesssim  \sum_{e \in \Eb} \| \tg  - \sum_{j=0}^{k} \frac {\delta^j}{j!} \dnh^j u \|_{0,e} (\|  \dnh\Pn(d_h) \|_{0,e}	+ h^{-1} \| d_h \|_{0,e} )
	 \lesssim\\
	h^{-1/2} \delta^{k+1}\| u \|_{k+1,\infty,\Oh} \vvvert d_h \vvvert_{\Oh}.	
	\end{multline*}
	Assembling the bounds for the seven terms we finally obtain
	\[
\vvvert \delta_h \vvvert_{\Oh}^2 \lesssim	h^{s-1} | u |_{s,\Oh} \vvvert \delta_h \vvvert_{\Oh} + h^{-1/2} \delta^{k+1} \| u \|_{k+1,\infty,\Omega}  \vvvert \delta_h \vvvert_{\Oh}.
	\]
	Dividing both sides by $\vvvert d_h \vvvert_{\Oh}$ and using a triangular inequality we get the thesis.
\end{proof}

As $\delta = o(h^2)$ and since we set $k = \lfloor m/2 \rfloor$ so that $h^{-1/2}\delta^{k+1} \lesssim h^{2 \lfloor m/2 \rfloor + 3/2} \lesssim h^m$ we have the following corollary.
	\begin{corollary}\label{cor:dh2}
Under the assumptions of Theorem \ref{main_theo}, if $u \in H^{m+1}(\Omega)$ then 
\[
\vvvert u - u_h \vvvert_{\Oh} \lesssim h^m \| u \|_{m+1,\Omega}.
\]
\end{corollary}

\newcommand{\map}{T}

\begin{remark} While, for the sake of simplicity, we assumed $\Omega$ to be convex,  our reasoning can be carried out to more general situations, provided $d_h = o(h)$. In particular, for non convex domains, Theorem \ref{main_theo} holds, provided $\Oh \subseteq \Omega$ (which, if $\Omega$ is not convex, requires to give up the assumption that the boundary vertices of $\Th$ belong to $\partial \Omega$). For an extension of the method where this assumption is relaxed, see \cite{Dupont_noconvex}.
\end{remark}

\begin{remark} In defining the method, we set $k = \lfloor m/2 \rfloor$. Of course, it is possible to choose other values for the parameter $k$ (for instance, if the condition $d_h=o(h^2)$ is not satisfied). Observe that, in the finite element case, the choice $k=m$ leads to the following discrete equation:
	\begin{align*}
	a(u_h,v_h) - \int_{\partial\Oh}\dnh u_h v_h - \int_{\partial\Oh}u^\star_h (\dnh v_h - \gamma h^{-1}v_h) = \\\int_{\Oh} f v_h - \int_{\partial \Oh} \tg  (\dnh v_h - \gamma h^{-1}v_h)
	\end{align*}
	where, for $x \in e \subset \Oh$, $u^\star_h(x) = p(x + \delta(x)\nuhx(x))$, $p$ being the polynomial in $\mathbb{P}_m$ such that $u_h = p$ in $K_e$ ($K_e$ denoting the only triangular element having $e$ as an edge). In our case we could expect that a similar property holds where $p$ is replaced by $\Pi^\nabla(u_h)$. However, this is not the case (at least, not exactly).  In fact, for $k = m$, if we rewrite \eqref{problem_BDT} in such a way to single out $\Pi^\nabla (u_h)^\star(x) = \Pi^\nabla_{K_e}(u_h)(x + \delta(x)\nuhx(x))$ we get
\begin{multline*}
 a_h(u_h,v_h) - \langle \dnh \Pn (u_h),v_h \rangle  \\ - \langle \Pi^\nabla(u_h)^\star, \dnh \Pn(v_h) - \gamma h^{-1} v_h \rangle+ \langle \Pn(u_h) - u_h , \dnh \Pn(v_h) - \gamma h^{-1} v_h  \rangle \\= \int_{\Oh} f v_h - \langle \tg, \dnh \Pn(v_h) - \gamma h^{-1} v_h \rangle,
	\end{multline*}
	which contains an extra term measuring the discrepancy between $u_h$ and $\Pi^\nabla(u_h)$ on $\partial\Oh$.
	\end{remark}


\section{Numerical Tests}\label{sec:expes}
In this section we present three different sets of numerical experiments, aimed at testing  and validating the proposed virtual element method for curved domains.
More precisely we deal with the following three different test cases, for each of which the right hand side $f$ and the boundary data $g$ are chosen in such a way that the solution to our model problem is the one given by, respectively, \eqref{sol1}, \eqref{sol2} and \eqref{sol3}.
\begin{enumerate}[\underline{Test} 1.]
\item $\Omega = \{(x,y)\in\mathbb R^2 | x^2 + y^2 \leq 1\}$.
 The analytic solution is given by
\begin{equation}\label{sol1}
u(x,y) = \cos(4\pi\sqrt{x^2+y^2}).
\end{equation}
\item $\Omega$ is the region bounded by the polar curve $x(\theta) = r(\theta)\cos(\theta), y(\theta) = r(\theta)\sin(\theta)$, with $r(\theta) = 2 + \sin(9\theta), \theta \in[0,2\pi]$.
 The analytic solution is given by
\begin{equation}\label{sol2}
u(x,y) = \sinc(2.25\sqrt{x^2+y^2})\cos(6.75\pi\sqrt{x^2+y^2}).
\end{equation}
\item $\Omega$ is the region bounded by the following curves:
\begin{align*}
&x(\theta) = r(\theta)\cos(\theta),\ y(\theta) = r(\theta)\sin(\theta), \text{ with } r(\theta) = \sqrt{\theta}, \quad\theta\in[\pi/2, 8\pi],\\
&x(\theta) = r(\theta)\cos(\theta),\ y(\theta) = r(\theta)\sin(\theta), \text{ with } r(\theta) = 0.9\sqrt{\theta}, \quad\theta\in[\pi/2, 8\pi],\\
&\{0\}\times[0.9\sqrt{\pi/2},\ \sqrt{\pi/2}], [0.9\sqrt{8\pi}, \sqrt{8\pi}]\times\{0\}. \\
\end{align*}
The analytic solution is given by 
\begin{equation} 	\label{sol3}
u(x,y) = \frac{\sin(32\tan^{-1}(y/x)) \cos(96\tan^{-1}(y/x))}{\sqrt{x^2 + y^2}}.
\end{equation}
\end{enumerate}
Figures  \ref{fig:circle_meshes}, \ref{fig:flower_meshes} and \ref{fig:spiral_meshes} display the three different domains considered (left)  and the computed VEM solution for $m=1$ (right).

\

Observe that, while the first of the three test cases falls under the assumptions under which we proved our  theoretical estimate, this is not the case for the second and third test cases, for both of which the domain $\Omega$ is not convex. In all three cases, the tessellations $\Th$ consist in quasi uniform shape regular Voronoi decompositions of the domain considered. As the grids are not structured, we choose to define the mesh size parameter as $h = N_V^{-1/2}$, where $N_V$ is the number of vertices of the tessellation. 

\

Letting $u_h^m$ denote the discrete solution obtained by the order $m$ VEM method proposed in the previous section, for all the three tests we consider the  relative error in the energy norm, as well as in the $L^2(\Oh)$ norm
\begin{align}\label{eS}
e^S_m &:= \frac{\vvvert u - u^m_h \vvvert_{\Oh}}{\vvvert u \vvvert_{\Oh}}, & e^{L^2} &:= \frac{|| u - u^m_h ||_{0,\Oh}}{|| u ||_{0,\Oh}}.
\end{align}

Tables \ref{tab:circle_meshes},  \ref{tab:flower_meshes} and \ref{tab:spiral_meshes} report $e^S_m$ for the three test cases, for $m=[1,\dots,6]$.  We also display in Figures \ref{fig:circle_error}, \ref{fig:flower_error}, \ref{fig:spiral_error} a logarithmic plot of the energy norm error $e^S_m$  and the $L^2$ norm $e^{L^2}_m$ as a function of the number of degrees of freedom $N_\text{DoFs}$ (which, we recall, is asymptotically proportional to $N_V$). 
The plots also show the approximate asymptotic convergence estimate obtained by plotting the functions $N_\text{DoFs}^{-m/2} \simeq h^{m}$. 

\

The numerical results for Test 1 are in agreement with the theoretical estimates. To test the robustness of the method we considered, in Tests 2 and 3, domains which are not convex; nevertheless the numerical results  are also in agreement with the theory and the predicted convergence rate of Corollary \ref{cor:dh2} is attained, see Tables \ref{tab:flower_meshes}, \ref{tab:spiral_meshes} and Figures  \ref{fig:flower_error} and \ref{fig:spiral_error}.

\begin{figure}[htb]
	\centering
	\subfloat{\includegraphics[width=0.45\textwidth]{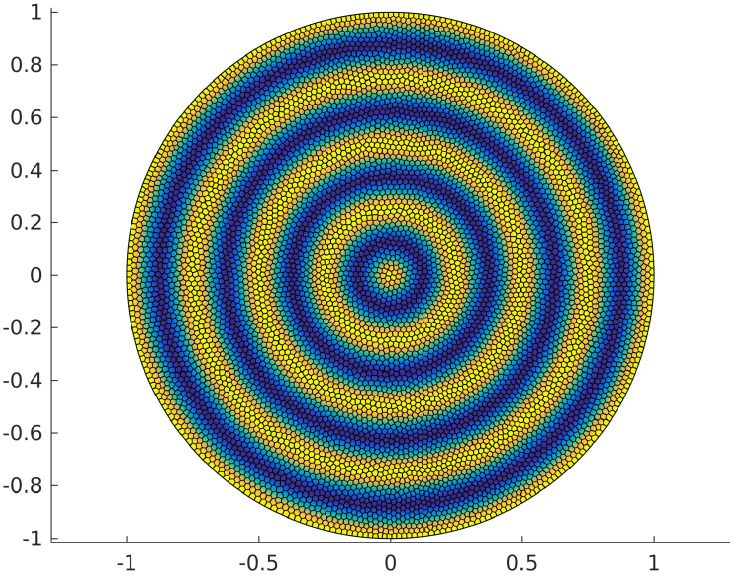}}
	\subfloat{\includegraphics[width=0.45\textwidth]{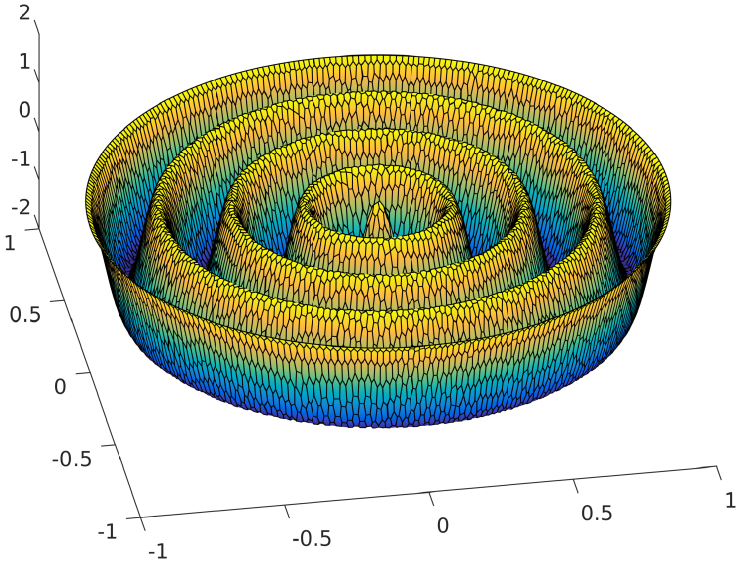}}
	\caption{First test case: (left) example mesh; (right) solution computed on the example mesh for $m = 1$.}
	\label{fig:circle_meshes}
\end{figure}

\begin{figure}[htb]
	\centering
	\subfloat{\includegraphics[width=0.45\textwidth]{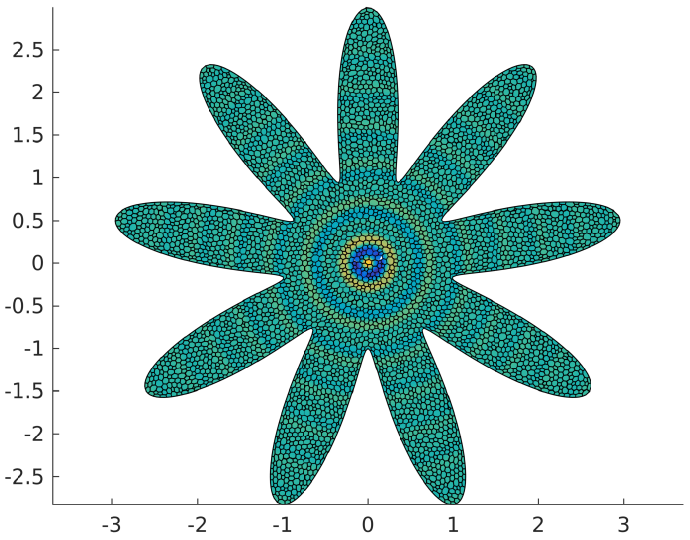}}
	\subfloat{\includegraphics[width=0.45\textwidth]{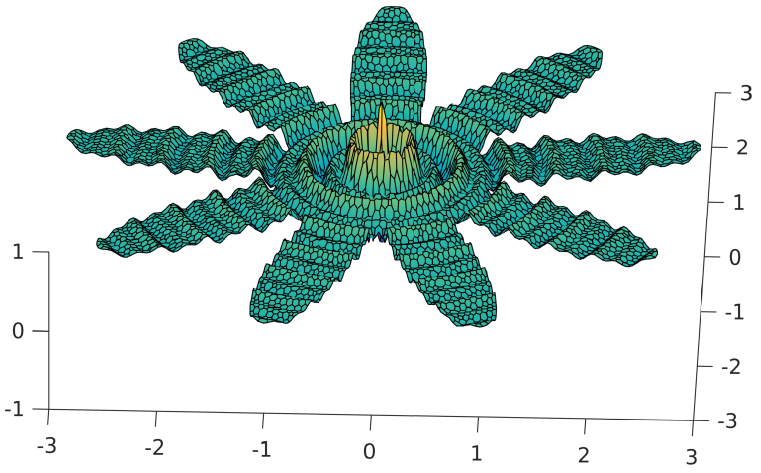}}
	\caption{Second test case: (left) example mesh; (right) solution computed on the example mesh for $m = 1$.}
	\label{fig:flower_meshes}
\end{figure}

\begin{figure}[htb]
	\centering
	\subfloat{\includegraphics[width=0.45\textwidth]{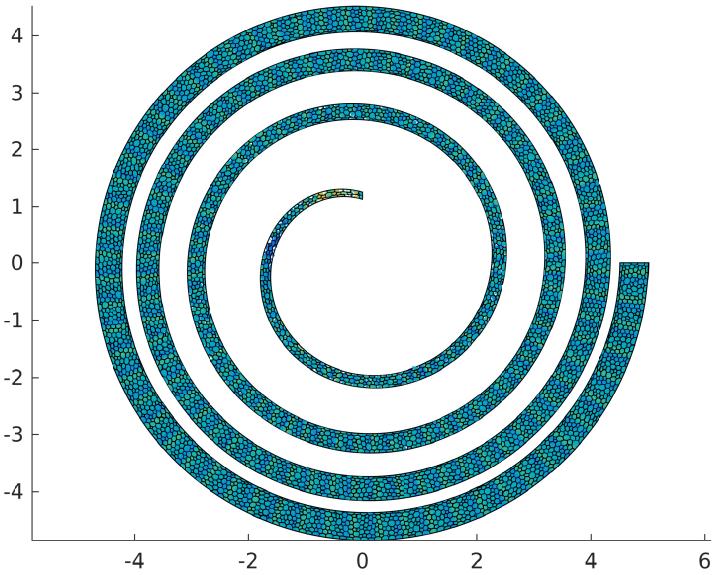}}
	\subfloat{\includegraphics[width=0.45\textwidth]{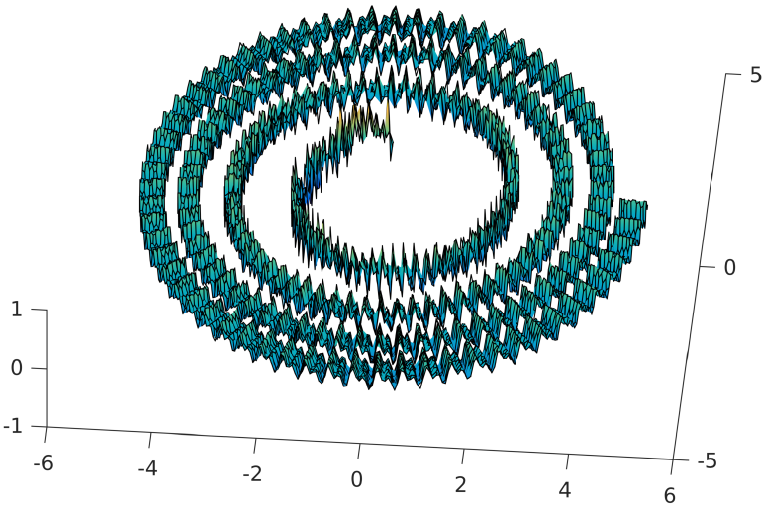}}
	\caption{Third test case: (left) example mesh; (right) solution computed on the example mesh for $m = 1$.}
	\label{fig:spiral_meshes}
\end{figure}


\begin{table}
\centering
\caption{First test case: relative errors in the energy norm for $m = 1,\dots,6.$}
\begin{tabular}{
c
S[table-format=1.{\roundPrecision}e-1]
S[table-format=1.{\roundPrecision}e-1]
S[table-format=1.{\roundPrecision}e-1]
S[table-format=1.{\roundPrecision}e-1]
S[table-format=1.{\roundPrecision}e-1]
S[table-format=1.{\roundPrecision}e-1]
S[table-format=1.{\roundPrecision}e-1]
}
\toprule
{Mesh} & {$h$} & {$e_1^S$} & {$e_2^S$} & {$e_3^S$} & {$e_4^S$} & {$e_5^S$} & {$e_6^S$}\\
\midrule
u-circle$_{1}$      &   4.42807e-02  &   5.545710e-02   &   2.193015e-02   &   1.199346e-02   &   5.036405e-03   &   7.509974e-04   &   1.789800e-04\\
u-circle$_{2}$      &   3.13728e-02   &   3.603052e-02   &   9.934135e-03   &   2.988500e-03   &   1.357457e-03   &   1.507302e-04   &   2.199918e-05\\
u-circle$_{3}$      &   2.21567e-02   &   2.088402e-02   &   3.915290e-03   &   7.734057e-04   &   3.160988e-04   &   2.343400e-05   &   2.404728e-06\\
u-circle$_{4}$      &   1.56671e-02   &   1.163849e-02   &   1.562648e-03   &   2.158015e-04   &   7.732008e-05   &   3.915458e-06   &   2.752169e-07\\
u-circle$_{5}$      &   1.10661e-02   &   6.347911e-03   &   6.028228e-04   &   5.491097e-05   &   1.587730e-05   &   5.857860e-07   &   2.961757e-08\\
u-circle$_{6}$      &   7.83188e-03   &   3.512339e-03   &   2.413337e-04   &   1.549042e-05   &   3.451427e-06   &   8.877508e-08   &   3.169268e-09\\
u-circle$_{7}$      &   5.53687e-03   &   2.022206e-03   &   9.817060e-05   &   4.342566e-06   &   7.429857e-07   &   1.353072e-08   &   3.398872e-10\\
u-circle$_{8}$      &   3.91405e-03   &   1.150340e-03   &   4.012099e-05   &   1.214472e-06   &   1.539748e-07   &   2.002459e-09   &   3.531635e-11\\
\bottomrule
\end{tabular}
\label{tab:circle_meshes}
\end{table}

\begin{figure}[htb]
\centering
\resizebox{\textwidth}{!}{
\begin{tabular}{rl}

\begin{tikzpicture}[trim axis left]

\begin{loglogaxis}[
legend style={
	font=\tiny},
legend pos=south west
]

\addplot [thick, solid, color=red, mark=o,mark options={scale=1.5}] coordinates {
(510, 0.0554571)
(1016, 0.0360305)
(2037, 0.020884)
(4074, 0.0116385)
(8166, 0.00634791)
(16303, 0.00351234)
(32619, 0.00202221)
(65275, 0.00115034)
};
\addlegendentry{$e_1^S$}

\addplot [thick, solid, color=blue, mark=+,mark options={scale=1.5}] coordinates {
(1531, 0.0219301)
(3055, 0.00993413)
(6121, 0.00391529)
(12243, 0.00156265)
(24523, 0.000602823)
(48989, 0.000241334)
(98005, 9.81706e-05)
(196085, 4.0121e-05)
};
\addlegendentry{$e_2^S$}

\addplot [thick, dashed, color=red, mark=o,mark options={scale=1.5,solid}] coordinates {
(510, 0.0442807)
(1016, 0.0313728)
(2037, 0.0221567)
(4074, 0.0156671)
(8166, 0.0110661)
(16303, 0.00783188)
(32619, 0.00553687)
(65275, 0.00391405)
};
\addlegendentry{$N_\text{DoFs}^{-1/2}$}

\addplot [ thick, dashed, color=blue, mark=+,mark options={scale=1.5,solid}] coordinates {
(1531, 0.000653168)
(3055, 0.000327332)
(6121, 0.000163372)
(12243, 8.16793e-05)
(24523, 4.0778e-05)
(48989, 2.04127e-05)
(98005, 1.02036e-05)
(196085, 5.09983e-06)
};
\addlegendentry{$N_\text{DoFs}^{-1}$}

\end{loglogaxis}
\end{tikzpicture}

&

\begin{tikzpicture}[trim axis right]
\begin{loglogaxis}[
legend style={font=\tiny},
legend pos=south west
]

\addplot [thick, solid, color=red, mark=o,mark options={scale=1.5}] coordinates {
(510, 0.424854)
(1016, 0.244116)
(2037, 0.122332)
(4074, 0.0614129)
(8166, 0.0308882)
(16303, 0.0153813)
(32619, 0.00772965)
(65275, 0.00384396)
};
\addlegendentry{$e_1^{L^2}$}

\addplot [thick, solid, color=blue, mark=+,mark options={scale=1.5}] coordinates {
(1531, 0.118358)
(3055, 0.043115)
(6121, 0.0154098)
(12243, 0.00553022)
(24523, 0.00194552)
(48989, 0.000688025)
(98005, 0.00024328)
(196085, 8.55962e-05)
};
\addlegendentry{$e_2^{L^2}$}

\addplot [thick, dashed, color=red, mark=o,mark options={scale=1.5,solid}] coordinates {
(510, 0.00196078)
(1016, 0.000984252)
(2037, 0.000490918)
(4074, 0.000245459)
(8166, 0.000122459)
(16303, 6.13384e-05)
(32619, 3.0657e-05)
(65275, 1.53198e-05)
};
\addlegendentry{$N_\text{DoFs}^{-1}$}

\addplot [thick, dashed, color=blue, mark=+,mark options={scale=1.5,solid}] coordinates {
(1531, 1.66931e-05)
(3055, 5.9222e-06)
(6121, 2.08817e-06)
(12243, 7.3819e-07)
(24523, 2.60399e-07)
(48989, 9.22257e-08)
(98005, 3.25932e-08)
(196085, 1.15168e-08)
};
\addlegendentry{$N_\text{DoFs}^{-3/2}$}

\end{loglogaxis}

\end{tikzpicture}

\\

\begin{tikzpicture}[trim axis left]

\begin{loglogaxis}[
legend style={font=\tiny},
legend pos=south west
]

\addplot [thick, solid, color=darkpastelgreen, mark=o,mark options={scale=1.5}] coordinates {
(2808, 0.0119935)
(5606, 0.0029885)
(11229, 0.000773406)
(22460, 0.000215801)
(44976, 5.4911e-05)
(89867, 1.54904e-05)
(179775, 4.34257e-06)
(359663, 1.21447e-06)
};
\addlegendentry{$e_3^S$}

\addplot [thick, solid, color=orange, mark=+,mark options={scale=1.5,solid}] coordinates {
(4341, 0.0050364)
(8669, 0.00135746)
(17361, 0.000316099)
(34725, 7.73201e-05)
(69525, 1.58773e-05)
(138937, 3.45143e-06)
(277929, 7.42986e-07)
(556009, 1.53975e-07)
};
\addlegendentry{$e_4^S$}

\addplot [thick, dashed, color=darkpastelgreen, mark=o,mark options={scale=1.5,solid}] coordinates {
(2808, 6.72054e-06)
(5606, 2.38243e-06)
(11229, 8.40405e-07)
(22460, 2.97088e-07)
(44976, 1.0484e-07)
(89867, 3.71193e-08)
(179775, 1.31192e-08)
(359663, 4.63614e-09)
};
\addlegendentry{$N_\text{DoFs}^{-3/2}$}

\addplot [thick, dashed, color=orange, mark=+,mark options={scale=1.5,solid}] coordinates {
(4341, 5.30665e-08)
(8669, 1.33064e-08)
(17361, 3.3178e-09)
(34725, 8.29307e-10)
(69525, 2.0688e-10)
(138937, 5.18041e-11)
(277929, 1.29459e-11)
(556009, 3.23472e-12)
};
\addlegendentry{$N_\text{DoFs}^{-2}$}

\end{loglogaxis}
\end{tikzpicture}

&

\begin{tikzpicture}[trim axis right]
\begin{loglogaxis}[
legend style={font=\tiny},
legend pos=south west
]

\addplot [thick, solid, color=darkpastelgreen, mark=o,mark options={scale=1.5,solid}] coordinates {
(2808, 0.0229876)
(5606, 0.00637068)
(11229, 0.00158441)
(22460, 0.000399461)
(44976, 0.000101019)
(89867, 2.52672e-05)
(179775, 6.40441e-06)
(359663, 1.58621e-06)
};
\addlegendentry{$e_3^{L^2}$}

\addplot [thick, solid, color=orange, mark=+,mark options={scale=1.5,solid}] coordinates {
(4341, 0.00465633)
(8669, 0.000940123)
(17361, 0.000179126)
(34725, 3.54532e-05)
(69525, 5.9435e-06)
(138937, 1.06954e-06)
(277929, 1.92518e-07)
(556009, 3.33234e-08)
};
\addlegendentry{$e_4^{L^2}$}

\addplot [thick, dashed, color=darkpastelgreen, mark=o,mark options={scale=1.5,solid}] coordinates {
(2808, 1.26825e-07)
(5606, 3.18195e-08)
(11229, 7.93082e-09)
(22460, 1.98235e-09)
(44976, 4.94354e-10)
(89867, 1.23822e-10)
(179775, 3.09415e-11)
(359663, 7.73052e-12)
};
\addlegendentry{$N_\text{DoFs}^{-2}$}

\addplot [thick, dashed, color=orange, mark=+,mark options={scale=1.5,solid}] coordinates {
(4341, 8.05426e-10)
(8669, 1.42915e-10)
(17361, 2.51804e-11)
(34725, 4.45035e-12)
(69525, 7.84599e-13)
(138937, 1.38981e-13)
(277929, 2.45564e-14)
(556009, 4.33806e-15)
};
\addlegendentry{$N_\text{DoFs}^{-5/2}$}

\end{loglogaxis}

\end{tikzpicture}

\\

\begin{tikzpicture}[trim axis left]

\begin{loglogaxis}[
legend style={font=\tiny},
xlabel={$N_\text{DoFs}$},
legend pos=south west
]

\addplot [thick, solid, color=cyan, mark=o,mark options={scale=1.5,solid}] coordinates {
(6130, 0.000750997)
(12244, 0.00015073)
(24517, 2.3434e-05)
(49038, 3.91546e-06)
(98170, 5.85786e-07)
(196199, 8.87751e-08)
(392467, 1.35307e-08)
(785123, 2.00246e-09)
};
\addlegendentry{$e_5^S$}

\addplot [thick, solid, color=olive, mark=+,mark options={scale=1.5,solid}] coordinates {
(8175, 0.00017898)
(16331, 2.19992e-05)
(32697, 2.40473e-06)
(65399, 2.75217e-07)
(130911, 2.96176e-08)
(261653, 3.16927e-09)
(523389, 3.39887e-10)
(1047005, 3.53163e-11)
};
\addlegendentry{$e_6^S$}

\addplot [thick, dashed, color=cyan, mark=o,mark options={scale=1.5,solid}] coordinates {
(6130, 3.39898e-10)
(12244, 6.02826e-11)
(24517, 1.06251e-11)
(49038, 1.87788e-12)
(98170, 3.31172e-13)
(196199, 5.86487e-14)
(392467, 1.03632e-14)
(785123, 1.83086e-15)
};
\addlegendentry{$N_\text{DoFs}^{-5/2}$}

\addplot [thick, dashed, color=olive, mark=+,mark options={scale=1.5,solid}] coordinates {
(8175, 1.83036e-12)
(16331, 2.29595e-13)
(32697, 2.86073e-14)
(65399, 3.57509e-15)
(130911, 4.4573e-16)
(261653, 5.58242e-17)
(523389, 6.97471e-18)
(1047005, 8.71272e-19)
};
\addlegendentry{$N_\text{DoFs}^{-3}$}

\end{loglogaxis}
\end{tikzpicture}

&

\begin{tikzpicture}[trim axis right]
\begin{loglogaxis}[
legend style={font=\tiny},
xlabel={$N_\text{DoFs}$},
legend pos=south west
]

\addplot [thick, solid, color=cyan, mark=o,mark options={scale=1.5,solid}] coordinates {
(6130, 0.000794506)
(12244, 0.000118774)
(24517, 1.46724e-05)
(49038, 1.9438e-06)
(98170, 2.35811e-07)
(196199, 2.95957e-08)
(392467, 3.75659e-09)
(785123, 4.65383e-10)
};
\addlegendentry{$e_5^{L^2}$}

\addplot [thick, solid, color=olive, mark=+,mark options={scale=1.5,solid}] coordinates {
(8175, 0.000180169)
(16331, 1.58643e-05)
(32697, 1.32507e-06)
(65399, 1.16375e-07)
(130911, 1.02872e-08)
(261653, 8.84295e-10)
(523389, 7.90879e-11)
(1047005, 1.48083e-11)
};
\addlegendentry{$e_6^{L^2}$}

\addplot [thick, dashed, color=cyan, mark=o,mark options={scale=1.5,solid}] coordinates {
(6130, 4.34129e-12)
(12244, 5.44791e-13)
(24517, 6.78575e-14)
(49038, 8.48012e-15)
(98170, 1.05697e-15)
(196199, 1.32407e-16)
(392467, 1.65421e-17)
(785123, 2.06627e-18)
};
\addlegendentry{$N_\text{DoFs}^{-3}$}

\addplot [thick, dashed, color=olive, mark=+,mark options={scale=1.5,solid}] coordinates {
(8175, 2.02438e-14)
(16331, 1.79662e-15)
(32697, 1.58206e-16)
(65399, 1.39798e-17)
(130911, 1.23192e-18)
(261653, 1.09134e-19)
(523389, 9.64082e-21)
(1047005, 8.5149e-22)
};
\addlegendentry{$N_\text{DoFs}^{-7/2}$}

\end{loglogaxis}

\end{tikzpicture}

\\

\end{tabular}
}
\caption{First test case: 
	logarithmic plot of the energy norm error $e^S_m$ for increasing polynomial order $m=1,\dots,6$ and of the functions $N_\text{DoFs}^{-m/2}$ (dashed lines). }
\label{fig:circle_error}
\end{figure}
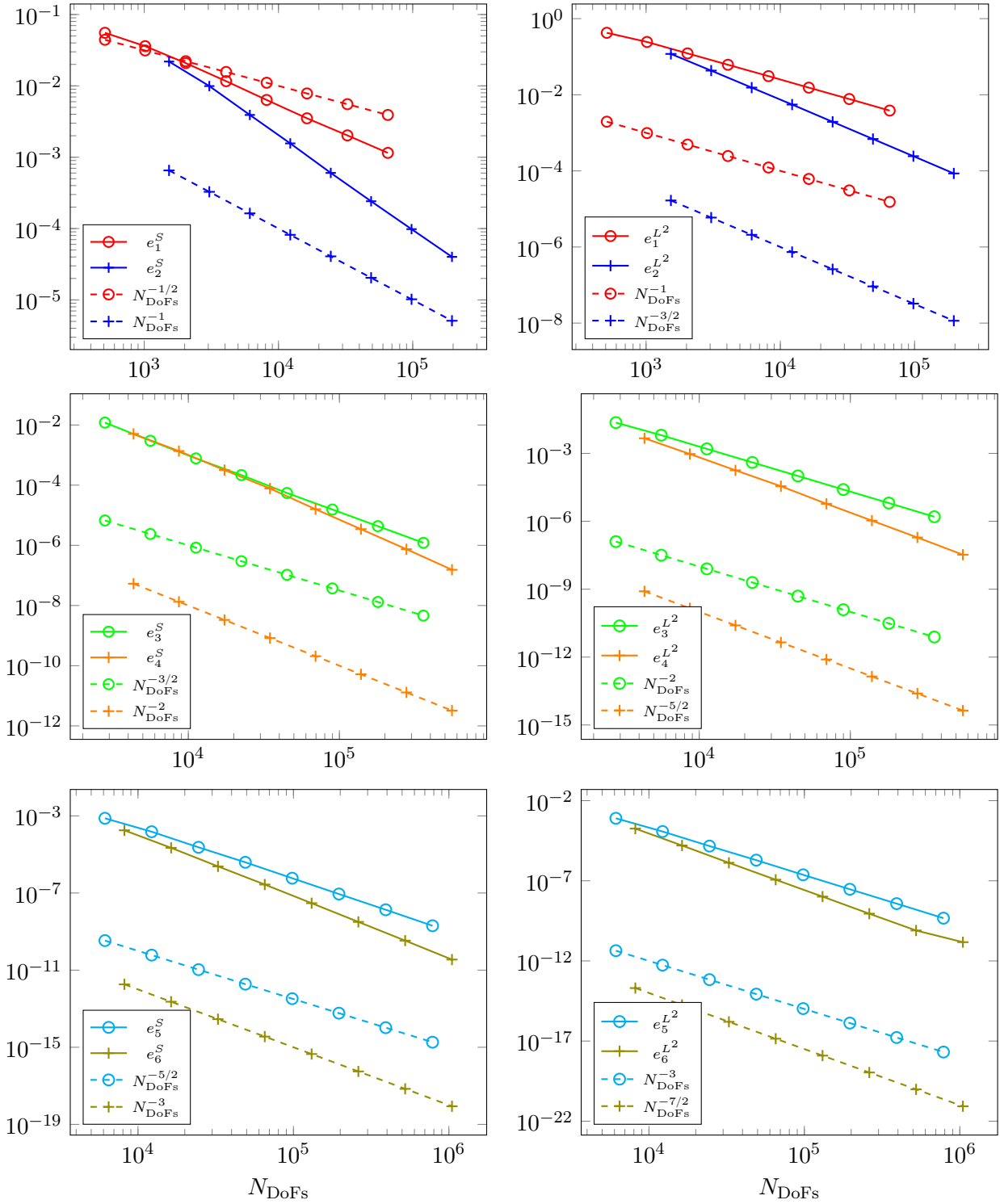

%
\begin{table}
	\centering
	\caption{Second test case: relative errors in the energy norm for $m = 1,\dots,6.$}
	\begin{tabular}{
			c
			S[table-format=1.{\roundPrecision}e-1]
			S[table-format=1.{\roundPrecision}e-1]
			S[table-format=1.{\roundPrecision}e-1]
			S[table-format=1.{\roundPrecision}e-1]
			S[table-format=1.{\roundPrecision}e-1]
			S[table-format=1.{\roundPrecision}e-1]
			S[table-format=1.{\roundPrecision}e-1]
		}
		\toprule
		{Mesh} & {$h$} & {$e_1^S$} & {$e_2^S$} & {$e_3^S$} & {$e_4^S$} & {$e_5^S$} & {$e_6^S$}\\
		\midrule
		flower$_{1}$      &   1.404e-02   &   5.502791e-02   &   3.018235e-02   &   1.065531e-02   &   4.716900e-03   &   9.854027e-04   &   2.017915e-04\\
		flower$_{2}$      &   1.01321e-02   &   3.402489e-02   &   1.294966e-02   &   2.952840e-03   &   1.201981e-03   &   1.325007e-04   &   2.388022e-05\\
		flower$_{3}$      &   7.26222e-03   &   1.937636e-02   &   6.016640e-03   &   8.586167e-04   &   3.147605e-04   &   2.205693e-05   &   4.133147e-06\\
		flower$_{4}$      &   5.18978e-03   &   1.238208e-02   &   2.527811e-03   &   2.568508e-04   &   7.075906e-05   &   3.637972e-06   &   3.787329e-07\\
		flower$_{5}$      &   3.6997e-03   &   7.158825e-03   &   1.117230e-03   &   7.643617e-05   &   1.658088e-05   &   5.828411e-07   &   4.556623e-08\\
		flower$_{6}$      &   2.63071e-03   &   4.267117e-03   &   4.710592e-04   &   2.369251e-05   &   3.643843e-06   &   9.018944e-08   &   4.826993e-09\\
		flower$_{7}$      &   1.86668e-03   &   2.652647e-03   &   2.019633e-04   &   7.107940e-06   &   7.864263e-07   &   1.379449e-08   &   5.601901e-10\\
		\bottomrule
	\end{tabular}
	\label{tab:flower_meshes}
\end{table}

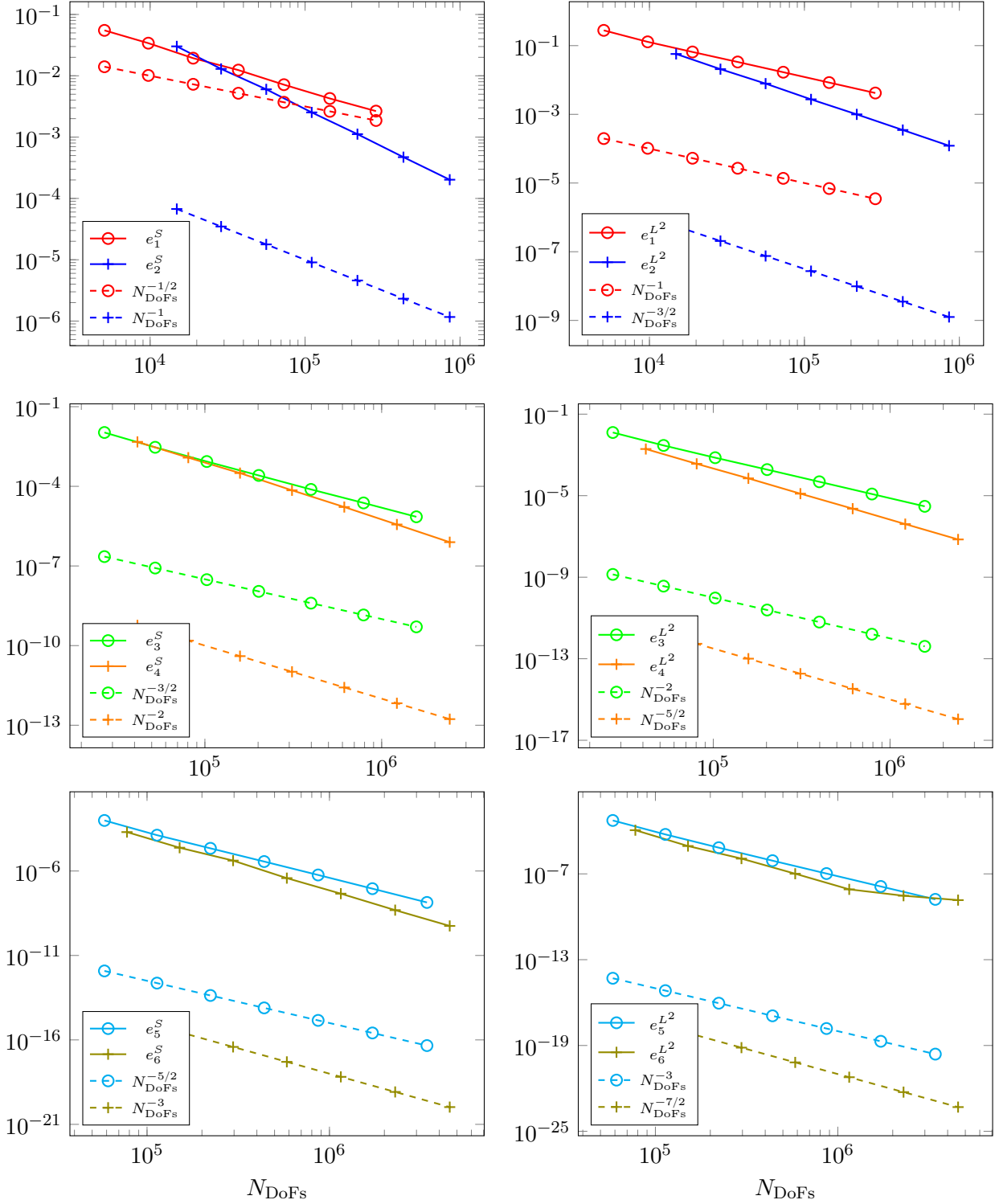
\begin{figure}[htb]
\centering
\resizebox{\textwidth}{!}{
\begin{tabular}{rl}

\begin{tikzpicture}[trim axis left]

\begin{loglogaxis}[
legend style={font=\tiny},
legend pos=south west
]

\addplot [thick, solid, color=red, mark=o,,mark options={scale=1.5,solid}] coordinates {
(5073, 0.0550279)
(9741, 0.0340249)
(18961, 0.0193764)
(37128, 0.0123821)
(73058, 0.00715883)
(144495, 0.00426712)
(286985, 0.00265265)
};
\addlegendentry{$e_1^S$}

\addplot [thick, solid, color=blue, mark=+,,mark options={scale=1.5,solid}] coordinates {
(14867, 0.0301824)
(28727, 0.0129497)
(56179, 0.00601664)
(110387, 0.00252781)
(217763, 0.00111723)
(431511, 0.000471059)
(858133, 0.000201963)
};
\addlegendentry{$e_2^S$}

\addplot [thick, dashed, color=red, mark=o, ,mark options={scale=1.5,solid}] coordinates {
(5073, 0.01404)
(9741, 0.0101321)
(18961, 0.00726222)
(37128, 0.00518978)
(73058, 0.0036997)
(144495, 0.00263071)
(286985, 0.00186668)
};
\addlegendentry{$N_\text{DoFs}^{-1/2}$}

\addplot [thick, dashed, color=blue, mark=+, ,mark options={scale=1.5,solid}] coordinates {
(14867, 6.72631e-05)
(28727, 3.48105e-05)
(56179, 1.78002e-05)
(110387, 9.05904e-06)
(217763, 4.59215e-06)
(431511, 2.31744e-06)
(858133, 1.16532e-06)
};
\addlegendentry{$N_\text{DoFs}^{-1}$}

\end{loglogaxis}
\end{tikzpicture}

&

\begin{tikzpicture}[trim axis right]
\begin{loglogaxis}[
legend style={font=\tiny},
legend pos=south west
]

\addplot [thick, solid, color=red, mark=o,,mark options={scale=1.5,solid}] coordinates {
(5073, 0.278327)
(9741, 0.128952)
(18961, 0.0654322)
(37128, 0.0334381)
(73058, 0.0168463)
(144495, 0.00841131)
(286985, 0.00417643)
};
\addlegendentry{$e_1^{L^2}$}

\addplot [thick, solid, color=blue, mark=+,,mark options={scale=1.5,solid}] coordinates {
(14867, 0.0571816)
(28727, 0.0206878)
(56179, 0.00784525)
(110387, 0.00272675)
(217763, 0.000997427)
(431511, 0.000349535)
(858133, 0.000121855)
};
\addlegendentry{$e_2^{L^2}$}

\addplot [thick, dashed, color=red, mark=o,,mark options={scale=1.5,solid}] coordinates {
(5073, 0.000197122)
(9741, 0.000102659)
(18961, 5.27398e-05)
(37128, 2.69339e-05)
(73058, 1.36878e-05)
(144495, 6.92065e-06)
(286985, 3.4845e-06)
};
\addlegendentry{$N_\text{DoFs}^{-1}$}

\addplot [thick, dashed, color=blue, mark=+, ,mark options={scale=1.5,solid}] coordinates {
(14867, 5.51652e-07)
(28727, 2.05383e-07)
(56179, 7.50998e-08)
(110387, 2.72661e-08)
(217763, 9.84065e-09)
(431511, 3.52787e-09)
(858133, 1.25796e-09)
};
\addlegendentry{$N_\text{DoFs}^{-3/2}$}

\end{loglogaxis}

\end{tikzpicture}

\\

\begin{tikzpicture}[trim axis left]

\begin{loglogaxis}[
legend style={font=\tiny},
legend pos=south west
]

\addplot [thick, solid, color=darkpastelgreen, mark=o,mark options={scale=1.5,solid}] coordinates {
(27022, 0.0106553)
(52336, 0.00295284)
(102526, 0.000858617)
(201712, 0.000256851)
(398292, 7.64362e-05)
(789788, 2.36925e-05)
(1571363, 7.10794e-06)
};
\addlegendentry{$e_3^S$}

\addplot [thick, solid, color=orange, mark=+, ,mark options={scale=1.5,solid}] coordinates {
(41538, 0.0047169)
(80568, 0.00120198)
(158002, 0.00031476)
(311103, 7.07591e-05)
(614645, 1.65809e-05)
(1219326, 3.64384e-06)
(2426675, 7.86426e-07)
};
\addlegendentry{$e_4^S$}

\addplot [thick, dashed, color=darkpastelgreen, mark=o, ,mark options={scale=1.5,solid}] coordinates {
(27022, 2.25125e-07)
(52336, 8.35217e-08)
(102526, 3.04613e-08)
(201712, 1.10383e-08)
(398292, 3.9783e-09)
(789788, 1.42474e-09)
(1571363, 5.07674e-10)
};
\addlegendentry{$N_\text{DoFs}^{-3/2}$}

\addplot [thick, dashed, color=orange, mark=+, ,mark options={scale=1.5,solid}] coordinates {
(41538, 5.79574e-10)
(80568, 1.54055e-10)
(158002, 4.00567e-11)
(311103, 1.03322e-11)
(614645, 2.64698e-12)
(1219326, 6.72605e-13)
(2426675, 1.69815e-13)
};
\addlegendentry{$N_\text{DoFs}^{-2}$}

\end{loglogaxis}
\end{tikzpicture}

&

\begin{tikzpicture}[trim axis right]
\begin{loglogaxis}[
legend style={font=\tiny},
legend pos=south west
]

\addplot [thick, solid, color=darkpastelgreen, mark=o,mark options={scale=1.5,solid}] coordinates {
(27022, 0.0125813)
(52336, 0.00291413)
(102526, 0.000729562)
(201712, 0.000190187)
(398292, 4.80743e-05)
(789788, 1.20552e-05)
(1571363, 2.99981e-06)
};
\addlegendentry{$e_3^{L^2}$}

\addplot [thick, solid, color=orange, mark=+,mark options={scale=1.5,solid}] coordinates {
(41538, 0.00194538)
(80568, 0.000365321)
(158002, 7.05736e-05)
(311103, 1.25492e-05)
(614645, 2.30131e-06)
(1219326, 4.03499e-07)
(2426675, 7.11754e-08)
};
\addlegendentry{$e_4^{L^2}$}

\addplot [thick, dashed, color=darkpastelgreen, mark=o,mark options={scale=1.5,solid}] coordinates {
(27022, 1.36951e-09)
(52336, 3.65089e-10)
(102526, 9.51332e-11)
(201712, 2.45774e-11)
(398292, 6.30372e-12)
(789788, 1.60317e-12)
(1571363, 4.04992e-13)
};
\addlegendentry{$N_\text{DoFs}^{-2}$}

\addplot [thick, dashed, color=orange, mark=+,mark options={scale=1.5,solid}] coordinates {
(41538, 2.84371e-12)
(80568, 5.42742e-13)
(158002, 1.00773e-13)
(311103, 1.85242e-14)
(614645, 3.37629e-15)
(1219326, 6.09117e-16)
(2426675, 1.09011e-16)
};
\addlegendentry{$N_\text{DoFs}^{-5/2}$}

\end{loglogaxis}

\end{tikzpicture}

\\

\begin{tikzpicture}[trim axis left]

\begin{loglogaxis}[
legend style={font=\tiny},
xlabel={$N_\text{DoFs}$},
legend pos=south west
]

\addplot [thick, solid, color=cyan, mark=o,mark options={scale=1.5,solid}] coordinates {
(58415, 0.000985403)
(113423, 0.000132501)
(222607, 2.20569e-05)
(438560, 3.63797e-06)
(866822, 5.82841e-07)
(1720125, 9.01894e-08)
(3424069, 1.37945e-08)
};
\addlegendentry{$e_5^S$}

\addplot [thick, solid, color=olive, mark=+,mark options={scale=1.5,solid}] coordinates {
(77653, 0.000201792)
(150901, 2.38802e-05)
(296341, 4.13315e-06)
(584083, 3.78733e-07)
(1154823, 4.55662e-08)
(2292185, 4.82699e-09)
(4563545, 5.6019e-10)
};
\addlegendentry{$e_6^S$}

\addplot [thick, dashed, color=cyan, mark=o,mark options={scale=1.5,solid}] coordinates {
(58415, 1.21252e-12)
(113423, 2.30806e-13)
(222607, 4.27713e-14)
(438560, 7.85104e-15)
(866822, 1.42947e-15)
(1720125, 2.57691e-16)
(3424069, 4.60939e-17)
};
\addlegendentry{$N_\text{DoFs}^{-5/2}$}

\addplot [thick, dashed, color=olive, mark=+,mark options={scale=1.5,solid}] coordinates {
(77653, 2.13563e-15)
(150901, 2.91021e-16)
(296341, 3.8426e-17)
(584083, 5.01853e-18)
(1154823, 6.49312e-19)
(2292185, 8.30331e-20)
(4563545, 1.05219e-20)
};
\addlegendentry{$N_\text{DoFs}^{-3}$}

\end{loglogaxis}
\end{tikzpicture}

&

\begin{tikzpicture}[trim axis right]
\begin{loglogaxis}[
legend style={font=\tiny},
xlabel={$N_\text{DoFs}$},
legend pos=south west
]

\addplot [thick, solid, color=cyan, mark=o,mark options={scale=1.5,solid}] coordinates {
(58415, 0.000542626)
(113423, 5.73499e-05)
(222607, 6.78269e-06)
(438560, 8.3722e-07)
(866822, 1.05252e-07)
(1720125, 1.31633e-08)
(3424069, 1.64809e-09)
};
\addlegendentry{$e_5^{L^2}$}

\addplot [thick, solid, color=olive, mark=+,mark options={scale=1.5,solid}] coordinates {
(77653, 0.000113372)
(150901, 8.81757e-06)
(296341, 1.16749e-06)
(584083, 1.02236e-07)
(1154823, 8.26173e-09)
(2292185, 2.96253e-09)
(4563545, 1.46011e-09)
};
\addlegendentry{$e_6^{L^2}$}

\addplot [thick, dashed, color=cyan, mark=o,mark options={scale=1.5,solid}] coordinates {
(58415, 5.0168e-15)
(113423, 6.85325e-16)
(222607, 9.06533e-17)
(438560, 1.18553e-17)
(866822, 1.53536e-18)
(1720125, 1.96481e-19)
(3424069, 2.49099e-20)
};
\addlegendentry{$N_\text{DoFs}^{-3}$}

\addplot [thick, dashed, color=olive, mark=+,mark options={scale=1.5,solid}] coordinates {
(77653, 7.66384e-18)
(150901, 7.49165e-19)
(296341, 7.05877e-20)
(584083, 6.56658e-21)
(1154823, 6.04221e-22)
(2292185, 5.48436e-23)
(4563545, 4.9254e-24)
};
\addlegendentry{$N_\text{DoFs}^{-7/2}$}

\end{loglogaxis}

\end{tikzpicture}

\\

\end{tabular}
}
\caption{Second test case: logarithmic plot of the energy norm error $e^S_m$ for increasing polynomial order $m=1,\dots,6$ and of the functions $N_\text{DoFs}^{-m/2}$ (dashed lines). }
\label{fig:flower_error}
\end{figure}


\begin{table}
	\centering
	\caption{Third test case: relative errors in the energy norm for $m = 1,\dots,6.$}
	\begin{tabular}{
			c
			S[table-format=1.{\roundPrecision}e-1]
			S[table-format=1.{\roundPrecision}e-1]
			S[table-format=1.{\roundPrecision}e-1]
			S[table-format=1.{\roundPrecision}e-1]
			S[table-format=1.{\roundPrecision}e-1]
			S[table-format=1.{\roundPrecision}e-1]
			S[table-format=1.{\roundPrecision}e-1]
		}
		\toprule
		{Mesh} & {$h$} & {$e_1^S$} & {$e_2^S$} & {$e_3^S$} & {$e_4^S$} & {$e_5^S$} & {$e_6^S$}\\
		\midrule
		spiral$_{1}$      &   1.226609e-02   &   2.131452e-01   &   1.945072e-01   &   1.9476e-01   &   1.540414e-01   &   1.584283e-01   &   1.432532e-01\\
		spiral$_{2}$      &   9.04838e-03   &   1.731372e-01   &   1.205959e-01   &   1.05604e-01   &   9.205270e-02   &   9.248680e-02   &   7.937706e-02\\
		spiral$_{3}$      &   6.63023e-03   &   9.402012e-02   &   5.735291e-02   &   4.91294e-02   &   3.554941e-02   &   2.414495e-02   &   1.591393e-02\\
		spiral$_{4}$      &   4.82523e-03   &   5.033507e-02   &   2.647245e-02   &   2.04633e-02   &   1.184688e-02   &   5.093357e-03   &   2.252170e-03\\
		spiral$_{5}$      &   3.47696e-03   &   2.303949e-02   &   1.423990e-02   &   6.37605e-03   &   2.961279e-03   &   7.498457e-04   &   2.865194e-04\\
		spiral$_{6}$      &   2.49798e-03   &   1.204840e-02   &   6.214234e-03   &   1.78195e-03   &   7.169268e-04   &   1.194759e-04   &   3.196001e-05\\
		spiral$_{7}$      &   1.78509e-03   &   6.809396e-03   &   2.537865e-03   &   4.7333e-04   &   1.569801e-04   &   1.571347e-05   &   3.081353e-06\\
		\bottomrule
	\end{tabular}
	\label{tab:spiral_meshes}
\end{table}

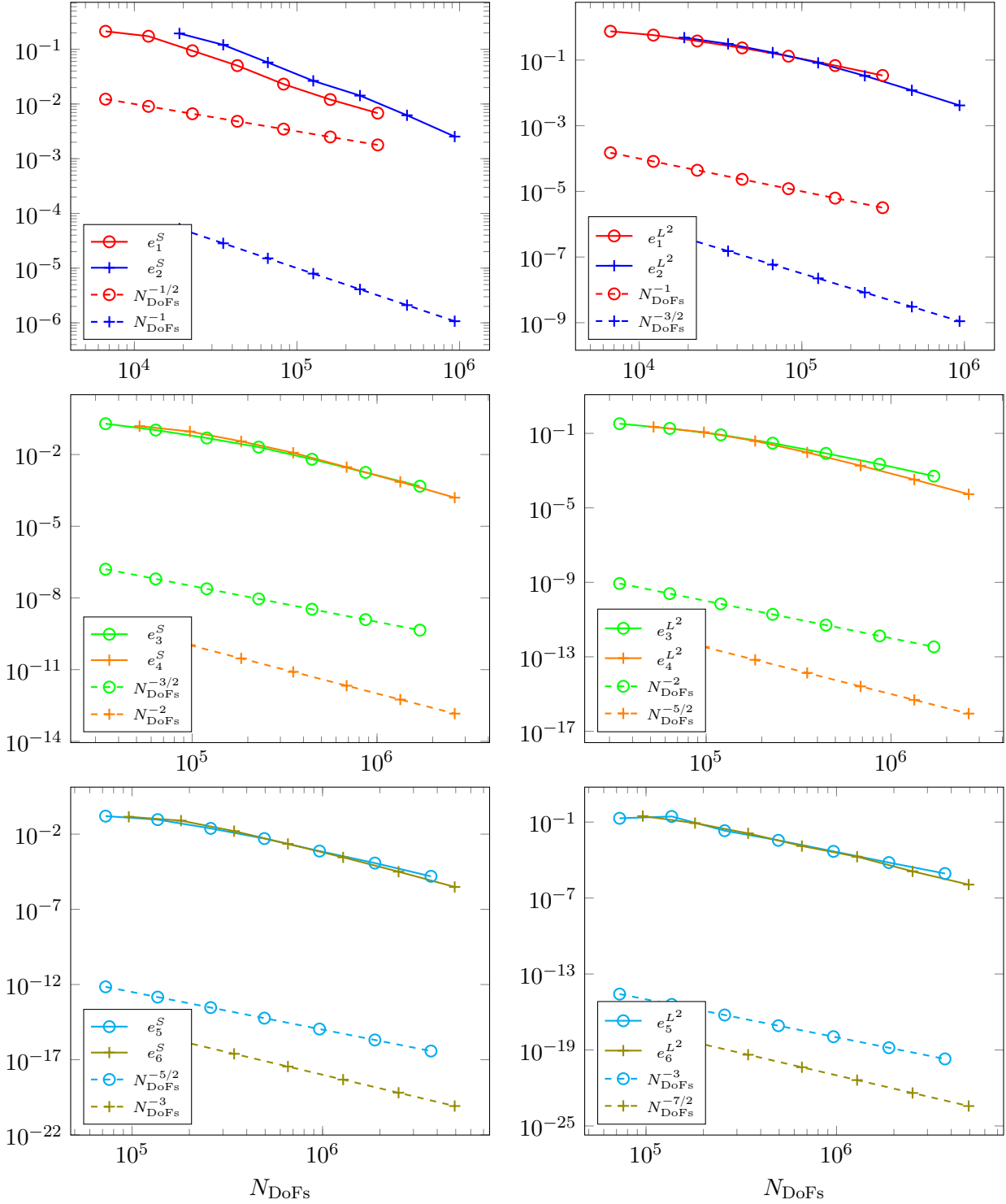
\begin{figure}[htb]
\centering
\resizebox{\textwidth}{!}{
\begin{tabular}{rl}

\begin{tikzpicture}[trim axis left]

\begin{loglogaxis}[
legend style={font=\tiny},
legend pos=south west
]

\addplot [thick, solid, color=red, mark=o,mark options={scale=1.5,solid}] coordinates {
(6652, 0.213145)
(12214, 0.173137)
(22748, 0.0940201)
(42950, 0.0503351)
(82718, 0.0230395)
(160259, 0.0120484)
(313819, 0.0068094)
};
\addlegendentry{$e_1^S$}

\addplot [thick, solid, color=blue, mark=+,mark options={scale=1.5,solid}] coordinates {
(18931, 0.194507)
(35183, 0.120596)
(66197, 0.0573529)
(125947, 0.0264724)
(244065, 0.0142399)
(474995, 0.00621423)
(933247, 0.00253787)
};
\addlegendentry{$e_2^S$}

\addplot [thick, dashed, color=red, mark=o,mark options={scale=1.5,solid}] coordinates {
(6652, 0.0122609)
(12214, 0.00904838)
(22748, 0.00663023)
(42950, 0.00482523)
(82718, 0.00347696)
(160259, 0.00249798)
(313819, 0.00178509)
};
\addlegendentry{$N_\text{DoFs}^{-1/2}$}

\addplot [thick, dashed, color=blue, mark=+,mark options={scale=1.5,solid}] coordinates {
(18931, 5.28234e-05)
(35183, 2.84228e-05)
(66197, 1.51064e-05)
(125947, 7.93985e-06)
(244065, 4.09727e-06)
(474995, 2.10529e-06)
(933247, 1.07153e-06)
};
\addlegendentry{$N_\text{DoFs}^{-1}$}

\end{loglogaxis}
\end{tikzpicture}

&

\begin{tikzpicture}[trim axis right]
\begin{loglogaxis}[
legend style={font=\tiny},
legend pos=south west
]

\addplot [thick, solid, color=red, mark=o,mark options={scale=1.5,solid}] coordinates {
(6652, 0.753253)
(12214, 0.576019)
(22748, 0.380509)
(42950, 0.234413)
(82718, 0.131399)
(160259, 0.068187)
(313819, 0.0339457)
};
\addlegendentry{$e_1^{L^2}$}

\addplot [thick, solid, color=blue, mark=+,mark options={scale=1.5,solid}] coordinates {
(18931, 0.482766)
(35183, 0.311475)
(66197, 0.169269)
(125947, 0.0823331)
(244065, 0.0332109)
(474995, 0.0119817)
(933247, 0.00415871)
};
\addlegendentry{$e_2^{L^2}$}

\addplot [thick, dashed, color=red, mark=o,mark options={scale=1.5,solid}] coordinates {
(6652, 0.000150331)
(12214, 8.18733e-05)
(22748, 4.39599e-05)
(42950, 2.32829e-05)
(82718, 1.20893e-05)
(160259, 6.2399e-06)
(313819, 3.18655e-06)
};
\addlegendentry{$N_\text{DoFs}^{-1}$}

\addplot [thick, dashed, color=blue, mark=+,mark options={scale=1.5,solid}] coordinates {
(18931, 3.83919e-07)
(35183, 1.51531e-07)
(66197, 5.87141e-08)
(125947, 2.23727e-08)
(244065, 8.29357e-09)
(474995, 3.05469e-09)
(933247, 1.10919e-09)
};
\addlegendentry{$N_\text{DoFs}^{-3/2}$}

\end{loglogaxis}

\end{tikzpicture}

\\

\begin{tikzpicture}[trim axis left]

\begin{loglogaxis}[
legend style={font=\tiny},
legend pos=south west
]

\addplot [thick, solid, color=darkpastelgreen, mark=o,mark options={scale=1.5,solid}] coordinates {
(34024, 0.19476)
(63530, 0.105604)
(119997, 0.0491294)
(228968, 0.0204633)
(444727, 0.00637605)
(866970, 0.00178195)
(1705480, 0.00047333)
};
\addlegendentry{$e_3^S$}

\addplot [thick, solid, color=orange, mark=+,mark options={scale=1.5,solid}] coordinates {
(51931, 0.154041)
(97255, 0.0920527)
(184148, 0.0355494)
(352013, 0.0118469)
(684704, 0.00296128)
(1336184, 0.000716927)
(2630518, 0.00015698)
};
\addlegendentry{$e_4^S$}

\addplot [thick, dashed, color=darkpastelgreen, mark=o,mark options={scale=1.5,solid}] coordinates {
(34024, 1.59339e-07)
(63530, 6.24499e-08)
(119997, 2.40572e-08)
(228968, 9.12721e-09)
(444727, 3.37178e-09)
(866970, 1.23878e-09)
(1705480, 4.48983e-10)
};
\addlegendentry{$N_\text{DoFs}^{-3/2}$}

\addplot [thick, dashed, color=orange, mark=+,mark options={scale=1.5,solid}] coordinates {
(51931, 3.70806e-10)
(97255, 1.05725e-10)
(184148, 2.94894e-11)
(352013, 8.07017e-12)
(684704, 2.13302e-12)
(1336184, 5.60102e-13)
(2630518, 1.44517e-13)
};
\addlegendentry{$N_\text{DoFs}^{-2}$}

\end{loglogaxis}
\end{tikzpicture}

&

\begin{tikzpicture}[trim axis right]
\begin{loglogaxis}[
legend style={font=\tiny},
legend pos=south west
]

\addplot [thick, solid, color=darkpastelgreen, mark=o,mark options={scale=1.5,solid}] coordinates {
(34024, 0.330481)
(63530, 0.183948)
(119997, 0.0832744)
(228968, 0.0298912)
(444727, 0.00845675)
(866970, 0.00221219)
(1705480, 0.000506112)
};
\addlegendentry{$e_3^{L^2}$}

\addplot [thick, solid, color=orange, mark=+,mark options={scale=1.5,solid}] coordinates {
(51931, 0.223275)
(97255, 0.115563)
(184148, 0.038979)
(352013, 0.00945892)
(684704, 0.00183905)
(1336184, 0.000328596)
(2630518, 5.39849e-05)
};
\addlegendentry{$e_4^{L^2}$}

\addplot [thick, dashed, color=darkpastelgreen, mark=o,mark options={scale=1.5,solid}] coordinates {
(34024, 8.63832e-10)
(63530, 2.47766e-10)
(119997, 6.94479e-11)
(228968, 1.90744e-11)
(444727, 5.05607e-12)
(866970, 1.33043e-12)
(1705480, 3.43801e-13)
};
\addlegendentry{$N_\text{DoFs}^{-2}$}

\addplot [thick, dashed, color=orange, mark=+,mark options={scale=1.5,solid}] coordinates {
(51931, 1.62717e-12)
(97255, 3.39016e-13)
(184148, 6.87199e-14)
(352013, 1.3602e-14)
(684704, 2.57776e-15)
(1336184, 4.84545e-16)
(2630518, 8.91039e-17)
};
\addlegendentry{$N_\text{DoFs}^{-5/2}$}

\end{loglogaxis}

\end{tikzpicture}

\\

\begin{tikzpicture}[trim axis left]

\begin{loglogaxis}[
legend style={font=\tiny},
xlabel={$N_\text{DoFs}$},
legend pos=south west
]

\addplot [thick, solid, color=cyan, mark=o,mark options={scale=1.5,solid}] coordinates {
(72652, 0.158428)
(136358, 0.0924868)
(258650, 0.0241449)
(495082, 0.00509336)
(963996, 0.000749846)
(1882637, 0.000119476)
(3708361, 1.57135e-05)
};
\addlegendentry{$e_5^S$}

\addplot [thick, solid, color=olive, mark=+,mark options={scale=1.5,solid}] coordinates {
(96187, 0.143253)
(180839, 0.0793771)
(343503, 0.0159139)
(658175, 0.00225217)
(1282603, 0.000286519)
(2506329, 3.196e-05)
(4939009, 3.08135e-06)
};
\addlegendentry{$e_6^S$}

\addplot [thick, dashed, color=cyan, mark=o,mark options={scale=1.5,solid}] coordinates {
(72652, 7.0288e-13)
(136358, 1.45646e-13)
(258650, 2.93913e-14)
(495082, 5.79839e-15)
(963996, 1.096e-15)
(1882637, 2.05629e-16)
(3708361, 3.77611e-17)
};
\addlegendentry{$N_\text{DoFs}^{-5/2}$}

\addplot [thick, dashed, color=olive, mark=+,mark options={scale=1.5,solid}] coordinates {
(96187, 1.1237e-15)
(180839, 1.69092e-16)
(343503, 2.46722e-17)
(658175, 3.50732e-18)
(1282603, 4.7394e-19)
(2506329, 6.35164e-20)
(4939009, 8.30005e-21)
};
\addlegendentry{$N_\text{DoFs}^{-3}$}

\end{loglogaxis}
\end{tikzpicture}

&

\begin{tikzpicture}[trim axis right]
\begin{loglogaxis}[
legend style={font=\tiny},
xlabel={$N_\text{DoFs}$},
legend pos=south west
]

\addplot [thick, solid, color=cyan, mark=o,mark options={scale=1.5,solid}] coordinates {
(72652, 0.199363)
(136358, 0.276779)
(258650, 0.02087)
(495082, 0.00361803)
(963996, 0.00048517)
(1882637, 6.32164e-05)
(3708361, 8.62977e-06)
};
\addlegendentry{$e_5^{L^2}$}

\addplot [thick, solid, color=olive, mark=+,mark options={scale=1.5,solid}] coordinates {
(96187, 0.295652)
(180839, 0.0828668)
(343503, 0.0128503)
(658175, 0.00128577)
(1282603, 0.000183502)
(2506329, 1.28395e-05)
(4939009, 1.15234e-06)
};
\addlegendentry{$e_6^{L^2}$}

\addplot [thick, dashed, color=cyan, mark=o,mark options={scale=1.5,solid}] coordinates {
(72652, 2.6077e-15)
(136358, 3.94419e-16)
(258650, 5.77913e-17)
(495082, 8.24079e-18)
(963996, 1.11628e-18)
(1882637, 1.49865e-19)
(3708361, 1.96089e-20)
};
\addlegendentry{$N_\text{DoFs}^{-3}$}

\addplot [thick, dashed, color=olive, mark=+,mark options={scale=1.5,solid}] coordinates {
(96187, 3.6232e-18)
(180839, 3.97629e-19)
(343503, 4.20962e-20)
(658175, 4.3232e-21)
(1282603, 4.18482e-22)
(2506329, 4.01205e-23)
(4939009, 3.73474e-24)
};
\addlegendentry{$N_\text{DoFs}^{-7/2}$}

\end{loglogaxis}

\end{tikzpicture}

\\

\end{tabular}
}
\caption{Third test case: logarithmic plot of the energy norm error $e^S_m$ for increasing polynomial order $m=1,\dots,6$ and of the functions $N_\text{DoFs}^{-m/2}$ (dashed lines). }
\label{fig:spiral_error}
\end{figure}

\bibliographystyle{amsplain}


\end{document}